\pdfoutput=1

\documentclass[a4paper]{amsart}

\usepackage{a4wide}
\usepackage[dvips]{graphicx}
\usepackage{color}
\usepackage{amssymb}
\usepackage{textcomp}

\let\oldtocsection=\tocsection

\let\oldtocsubsection=\tocsubsection

\let\oldtocsubsubsection=\tocsubsubsection

\renewcommand{\tocsection}[2]{\hspace{0em}\oldtocsection{#1}{#2}}
\renewcommand{\tocsubsection}[2]{\hspace{1em}\oldtocsubsection{#1}{#2}}
\renewcommand{\tocsubsubsection}[2]{\hspace{2em}\oldtocsubsubsection{#1}{#2}}

\newcounter{notes}

\usepackage{setspace,graphicx}     
\usepackage[square]{natbib}                    
\usepackage{stackrel}                                          
\usepackage{amssymb}
\usepackage{verbatim} 

\usepackage{blindtext}
\usepackage{tabularx}
\usepackage{longtable}

\usepackage{amsmath} 
\usepackage{amsthm} 
\usepackage{mathrsfs} 
\usepackage{amsfonts} 
\usepackage{latexsym} 
\usepackage{amscd} 
\usepackage[latin1]{inputenc} 
\usepackage[english]{babel} 
\usepackage{enumerate} 
\usepackage{afterpage} 
\allowdisplaybreaks

\def\co{\colon\thinspace} 
\def\a{\alpha}

\def\e{\epsilon} 
\def\g{\gamma} 
 
\def\l{\lambda}

\def\s{\sigma} 
\def\t{\tau} 
\def\th{\theta} 
\def\cal{\mathcal} 
\def\CC{\mathcal C} 
 
\def\FN{\mathrm{FN}} 
\def\Chat{\hat {\mathbb C}}
\def\E{{\cal E}} 
\def\F{{\cal F}} 
 
\def\H{{\cal H}} 
 
\def\M{{\cal M}} 
\def\S{{\cal S}} 
 
\def\DD{{\cal D}}

\def\QF{{\mathcal {QF}}} 
 
\def\L{{\cal L}} 
\def\P{{\cal P}} 
\def\PC{\mathcal {PC}} 

\def\FN{\mathrm{FN}}

\def\pl{\mathrm{pl}}

\def\T{{\cal T}} 
 
\newcommand{\C}{\mathbb{C}} 
\newcommand{\A}{\mathbb{A}} 
 
\newcommand{\R}{\mathbb{R}} 
\newcommand{\LL}{\mathbb{L}} 
 
\newcommand{\Z}{\mathbb{Z}}

\newcommand{\D}{\mathbb{D}} 
\newcommand{\HH}{\mathbb{H}} 
 
\def\Tr{\mathop{\rm Tr}} 
\def\Fix{\mathop{\rm Fix}} 
\def\Ax{\mathop{\rm Axis}} 
 
\def\Dev{\mathop{\rm Dev}} 
\def\AH{\mathop{\rm AH}}

\def\SL{\mathop{\rm SL}} 
\def\PSL{\mathop{\rm PSL}}
\def\Id{\mathop{\rm Id}} 
\renewcommand{\to}{\longrightarrow} 
\newcommand{\dd}{\partial}
 
\def\Arg{\mathop{\rm{Arg}}}

\def\ML{\mathop{\rm ML}} 
 
\def\PML{\mathop{\rm PML}}

\def\QF{{\cal {QF}}} 
\def\sech{\mathop{\rm{sech}}} 
 
\def\Tw{\mathop { {Tw}}}

\newcommand{\BM}{\mathcal{BM}} 
\def\teich{{\cal T}}

\def\ull{{\underline l}} 
\def\utt{{\underline t}}
\def\utb{{\underline{\textbf{t}}}}
 
\def\um{{\underline \mu}} 
\def\umb{{\underline{\boldsymbol\mu}}}

\def\ul{{\underline \l}} 
\def\ut{{\underline \t}}

\def\uc{{\underline c}}

\AtBeginDocument{%
   \def\MR#1{}
}

\newtheorem{Theorem}{Theorem}[section]
\newtheorem{Lemma}[Theorem]{Lemma}
\newtheorem{Proposition}[Theorem]{Proposition}
\newtheorem{Conjecture}[Theorem]{Conjecture}
\newtheorem{Corollary}[Theorem]{Corollary}
\newtheorem{introthm}{Theorem}

\newtheorem*{ConjectureNO}{Conjecture}
\theoremstyle{definition}
\newtheorem{Definition}[Theorem]{Definition}  
\theoremstyle{remark}
\newtheorem{Remark}[Theorem]{Remark}

\begin{document}

\title{c--gluing construction and slices of quasi-Fuchsian space}

\author{Sara Maloni}
\address{Department of Mathematics, University of Virginia}
\email{sm4cw@virginia.edu}
\urladdr{http:/www.people.virginia.edu/$\sim$sm4cw}

\thanks{The author was partially supported by the National Science Foundation under the grant DMS-1506920 and DMS-1650811, and by U.S. National Science Foundation grants DMS-1107452, 1107263, 1107367 ``RNMS: GEometric structures And Representation varieties'' (the GEAR Network).} 

\begin{abstract}
Given a pants decomposition $\mathcal{PC} = \{\gamma_1, \ldots, \gamma_{\xi}\}$ on a hyperbolizable surface $\Sigma$ and a vector $\uc = (c_1, \ldots, c_{\xi}) \in \R_+^\xi$, we describe a plumbing construction  which endows $\Sigma$ with a complex projective structure for which the associated holonomy representation $\rho$ is quasi-Fuchsian and for which $\ell_\rho(\gamma_i) = c_i$. When $\uc \to \underline{0} = (0, \ldots, 0)$ this construction limits to Kra's plumbing construction. In addition, when $\Sigma = \Sigma_{1,1}$, the holonomy representations of these structures belong to the `linear slice' of quasi-Fuchsian space $\QF(\Sigma)$ defined by Komori and Parkonnen. We discuss some conjectures for these slices suggested by the pictures we created in joint work with Yamashita. 
\end{abstract}
\maketitle 

\tableofcontents

\section{Introduction} \label{sec:introduction} 

Given a closed orientable surface $\Sigma$ of genus $g\geq 2$, the space $\AH(\Sigma)$ of hyperbolic $3$--manifolds homotopy equivalent to $\Sigma \times [0,1]$ can be naturally identified with the space of discrete and faithful representations from the surface group $\pi_1(\Sigma)$ into $\PSL(2, \C)$. While the interior of $\AH(\Sigma)$ has been well-understood since the 1970s and consists of quasi-Fuchsian (or convex-cocompact) representations, the topology of $\AH(\Sigma)$ is much less well-behaved. In fact, it is not even locally connected \cite{bro_the, mag_def}. Understanding the topology of $\AH(\Sigma)$  is quite difficult, so one can focus on some of its slice. In order to study these slices we need to define `good' coordinates for these deformation spaces, such as the ones provided by plumbing constructions. Plumbing constructions have been introduced in the 70's in order to describe holomorphic coordinates for moduli space of hyperbolizable surfaces which are intrinsic and which can be extended at infinity. Among these constructions, two of the most famous constructions are due to Kra \cite{kra_hor} and Earle and Marden \cite{ear_hol} and are associated to two well-known slices of important slices of $\AH(\Sigma)$: the Maskit slice and the Bers slice, respectively.

The construction discussed in this article is inspired by Kra's construction. The idea of Kra's plumbing construction is the following. Let $\Sigma$ be a closed orientable hyperbolizable surface and let $\PC = \{\s_1, \dots, \s_\xi\}$ be a pants decomposition of $\Sigma$. Kra's plumbing construction describes a complex projective structure on $\Sigma$ as follows. Identify each pair of pants with a thrice punctured sphere, truncate the pairs of pants along a horosphere and glue, or ``plumb'', the truncated pairs of pants along annuli homotopic to the punctures. The gluing across the i-th pants curve is defined by parameters $\mu_i\in \mathbb{C}$ which correspond to `horocyclic coordinates' in punctured disk neighbourhoods of the two punctures.  The holonomy representation $\rho\co \pi_1(\Sigma) \to \PSL(2,\C)$ associated with the above complex projective structure depends holomorphically on the $\mu_i$, and, by construction, the images of the pants curves are parabolic elements. In joint work with Series \cite{mal_top} we study a different and slightly simpler description of Kra's plumbing construction which respects the twisting around the puncture. This allows us to define coordinates in a particular slice of $\AH(\Sigma)$. Given an element $\gamma \in \pi_1(\Sigma)$ associated to a simple closed curve of $\Sigma$, the trace of $\rho(\gamma)$ is a polynomial in the $\mu_i$, and the main result of \cite{mal_top} is a relationship between the coefficients of the top terms of that polynomial and the Dehn-Thurston coordinates of $\gamma$ relative to $\PC$. If the developing map associated with the projective structure is an embedding, then the associated hyperbolic $3$--manifold $\HH^3/\rho(\pi_1(\Sigma))$ lies on the Maskit slice $\mathcal{M} = \mathcal{M}(\Sigma)$, the space of geometrically finite groups on the boundary of quasi-Fuchsian space $\mathcal{QF}(\Sigma)$ for which the `bottom' end consists of triply punctured spheres obtained from $\Sigma$ by pinching the pants curves in $\PC$. In \cite{mal_asy} using results from \cite{mal_top} and a careful analysis of the geometry of the convex core of the associated manifolds, we describe the asymptotic direction of pleating rays in $\mathcal{M}$ supported on multicurves. (Recall that given a projective measured lamination $[\eta]$ on $\Sigma$, the pleating ray $\mathcal{P} = \mathcal{P}_\eta$ is the set of representations in $\mathcal{M}$ for which the bending measure of the top component of the boundary of the convex core of the associated $3$--manifold is in $[\eta]\in \mathrm{PML}(\Sigma)$.)

In this article, we define a more general plumbing construction, called the $\uc$--plumbing construction, where $\uc = (c_1, \ldots, c_\xi)  \in (\R_{> 0})^\xi$. The idea is the following. Let $\PC = \{\s_1, \dots, \s_\xi\}$ be a pants decomposition  of a hyperbolizable surface $\Sigma = \Sigma_{g, b}$ of genus $g$ and $b$ punctures and with complexity $\xi = 3g-3+b>0$. Identify each pair of pants with a three-holed sphere so that the length of the boundary components corresponding to $\s_i$ is $2c_i$ for each $i = 1, \ldots, \xi$. Then truncate these pairs of pants by cutting  along annuli parallel to the boundary components and `plumb' adjacent pants along annuli parallel to the boundary components. The gluing is defined by complex parameters $\boldsymbol\mu_i \in \C_{[0,\pi)} = \{z \in \C \mid \mathrm{Im}z \in [0,\pi)\}$. This defines a complex projective structure on $\Sigma$ with holonomy representation $\rho_{\uc, \umb}\co\pi_1(\Sigma) \to \PSL(2, \C)$ and developing map $\mathrm{Dev}_{\uc, \umb}\co\widetilde{\Sigma} \to \mathbb{CP}^1$. A natural question is: what is the relationship between this construction and Kra's construction described above? In Section \ref{sec:lim} we show the following limiting behaviour. 

\begin{introthm}\label{thmE} 
  Let $\uc = (c_{1},\ldots, c_{\xi}) \in \R_{+}^{\xi}$ and $\umb = ({\boldsymbol\mu}_1, \ldots, {\boldsymbol\mu}_\xi)\in (\C_{[0, \pi)})^\xi$. If $\uc \to \underline{0}$ keeping $\um  = (\mu_1, \ldots, \mu_\xi)$ fixed, where $\mu_i = \frac{i\pi-{\boldsymbol\mu}_i}{c_i}$ (and $\Im{\mu}_i > 0$ for all $i = 1,\ldots, \xi$), then the complex structures $\Sigma(\uc, \umb)$ defined by the $\uc$--gluing construction with parameter $\umb$ limits to the complex projective structure $\Sigma(\um)$ defined by the gluing construction (of \cite{mal_top}) with parameter $\um = (\mu_1, \ldots, \mu_\xi)\in \HH^\xi$. 
\end{introthm}

When the developing map $\mathrm{Dev}_{\uc, \umb}$ is an embedding, we prove that the holonomy representation $\rho_{\uc, \umb}$ lies in the quasi-Fuchsian space and the length of the pants curves $\s_i$ is $2c_i$. In particular, when $\Sigma = \Sigma_{1,1}$ is a once-punctured torus (and $\PC = \{\s\}$ and $c>0$), the representation $\rho_{\uc, \umb}$ lies in the \textit{linear slice} $\L_{c}(\Sigma_{1,1})$ of the quasi-Fuchsian space $\QF(\Sigma_{1,1})$ as defined by Komori and Parkonnen~\cite{kom_ont}:
$$\L_{c}(\Sigma_{1,1}) = \{\t \in \C/2\pi i\Z \mid (c,\t) \in \FN_{\C}\left(\QF(\Sigma_{1,1})\right)\},$$
where $\FN_{\C} \co \QF(\Sigma_{1,1} \to \C_+/2\pi i\Z \times \C/2\pi i\Z$ is the (complex) Fenchel-Nielsen parametrization of $\QF(\Sigma_{1,1})$. This slice has a connected component, the \textit{Bers--Maskit slice} $\BM_c(\Sigma)$, containing the Fuchsian locus $\tau \in \R \cap \L_{c},$ and its points correspond to quasi-Fuchsian manifolds whose convex core is bent along $\sigma$. See also McMullen \cite{mcm_com}. Komori and Yamashita~\cite{kom_lin} proved that there exist two real constants $0 < C_0 < C_1$ such that, for any $0< c < C_0$, the linear slice coincides with the Bers--Maskit slice, while, for all $c > C_1$, the linear slice has many connected components. Together with Yamashita, we wrote a computer program which draws the slices $\L_c(\Sigma_{1,1})$ for different values of the parameter $c$, see Figure \ref{conn} and \ref{conn2}. In Section \ref{sub:extra_components} we describe our ideas about how to define the slice for a general hyperbolizable surface $\Sigma$ and how to generalise some of the results about its connected components. For example, among many questions and conjectures that Figure \ref{conn} and \ref{conn2} suggest, we want to underline the following conjecture. (Remember that the the total Maskit slice $\M^{tot}(\Sigma)$ is the space of geometrically finite groups on the boundary of quasi-Fuchsian space $\mathcal{QF}(\Sigma)$ for which one end is homeomorphic to $\Sigma$, while the other end consists of triply punctured spheres obtained by piching all the the pants curves in the pants decomposition $\PC$.)

\begin{ConjectureNO}[Conjecture \ref{union}]
  Given a hyperbolizable surface $\Sigma$ together with a pants decomposition $\PC$ and two positive numbers $c_1, c_2 \in \R_+$, we have the following:
  \begin{enumerate}
  \item If $c_1 \leqslant c_2,$ then $BM_{c_2} \subseteq BM_{c_1}$.
  \item $\M^{tot} = \cup_{c > 0} BM_{c}.$
  \end{enumerate}
\end{ConjectureNO}

We hope to explore further these slices in a future paper.

\subsection{Acknowledgements} We are grateful to Yasushi Yamashita to helping us drawing the pictures of Figures \ref{conn}, \ref{conn3}, \ref{conn2} and to Francesco Bonsante, Brian Bowditch, David Dumas, Brice Loustau, John Parker and Caroline Series for interesting conversations and helpful comments.  

\section{Background material} \label{sec:bacground}

\subsection{Curves on surfaces}

Suppose $\Sigma$ is a surface of finite type, let $\S_0 = \S_0(\Sigma)$ denote the set of free homotopy classes of connected closed simple non-trivial non-peripheral curves on $\Sigma$. Let $\S = \S(\Sigma)$ be the set of free homotopy classes of multicurves on $\Sigma$, where a \textit{multicurve} is a finite unions of disjoint simple closed curves in $\S_0$. The \textit{geometric intersection number} $i(\a,\a')$  between multicurves $\a,\a' \in \S$ is defined by $$i(\a, \a') = \min_{a \in \a , \; a' \in \a'} |a \cap a'|.$$ 

Given a surface $\Sigma = \Sigma_{g, b}$ of finite type (with genus $g$ and $b$ punctures) and negative Euler characteristic $\xi(\Sigma)$, choose a maximal set $\PC = \{ \s_{1}, \ldots, \s_{\xi}\}$ of homotopically distinct curves in $\Sigma$ called \emph{pants curves}, where  $\xi = \xi(\Sigma) = 3g-3+b$ is the complexity of the surface. These curves split the surface into $k=2g-2+b = -\chi(\Sigma)$ three-holed spheres $P_1,\ldots, P_k$, called \emph{pairs of pants}. (Note  that  the boundary of $P_i$ may include  punctures of $\Sigma$.)  We refer to both the set $\P = \{ P_1,\ldots, P_k\}$ and the set $\PC $ as \emph{pants decompositions} of  $\Sigma$. Any hyperbolic pair of pants $P$ is made by gluing two (maybe degenarate) right angled hexagons along three alternate edges which we call its \emph{seams}. We will consider \emph{dual curves} $D_i$ to the pants curves $\s_i \in \PC$, that is, curves which intersect $\s_i$ minimally and such that $i(D_i, \sigma_j) = 0$ for all $j\neq i$. 
\subsubsection{Fenchel--Nielsen twist deformation}\label{sec:dehntwists}

Our convention is always to consider twists to the right as positive. In particular, given a surface $\Sigma$, a curve $\s \in \S_0$ and $t \in \R$, the distance $t$ \emph{(right) Fenchel--Nielsen twist deformation} around $\s$ is the homeomorphism ${\Tw}_{\s, t}\co \Sigma \to \Sigma$ defined in the following way. Let $\A = \A(\s) = \s \times [0,1]$ be a (small) embedded annulus around $\s$. If we parameterise $\s$ as $s \mapsto \s(s) \in \Sigma$ for $s \in [0,1)$, then the distance $t$ twist, denoted ${\Tw}_{\s, t} \co \Sigma \to \Sigma$, maps $\A$ to itself by $(\s(s),\theta) \mapsto (\s(s+\theta t),\theta)$ and is the identity elsewhere. This definition extends to multicurves $\s \in \S$ by considering disjoint annuli around the curves in $\s$.

\subsubsection{Marking and marking decomposition}\label{sec:marking}

A \emph{marking} on $\Sigma$ is the specification of a fixed base (topological) surface $\Sigma_0$, together with a homeomorphism $f \co \Sigma_0 \to \Sigma$. 

There is a related notion of marking decomposition. Given the pants decomposition $\P$, we can fix a \textit{marking decomposition} on $\Sigma$ in two equivalent ways, see \cite{mal_sli}: 
\begin{enumerate}
\renewcommand{\labelenumi}{(\alph{enumi})}
    \item an \textit{involution}: an orientation--reversing map $\mathbf{R} \co \Sigma \to \Sigma$ so that for each $i = 1, \ldots, \xi$ we have $\mathbf{R}(\s_{i}) = \s_{i}$;
	\item \textit{dual curves}: for each $i$, a curve $D_{i}$ so that $i(D_{i}, \s_{j}) = 0$ if $i\neq j$ and $i(D_{i}, \s_{i})$ is minimal (so $i(D_{i}, \s_{i}) = 2$ if $\s_i$ is separating and $i(D_{i}, \s_{i}) = 1$ otherwise).
\end{enumerate}

A marking decomposition decomposes each pair of pants into two (possibly degenerate) hexagons.

\subsection{Measured laminations} \label{sub:meas_lam}

Given a surface $\Sigma$ endowed with an hyperbolic structure, a \textit{geodesic lamination} $\eta$ on $\Sigma$ is a closed set of pairwise disjoint complete simple geodesics on $\Sigma$ called its \textit{leaves}. A \textit{transverse measure} on $\eta$ is an assignment of a measure to each arc transverse to the leaves of $\eta$ that is invariant under the push forward maps along the leaves of $\eta$. A \textit{measured geodesic lamination} on $\Sigma$ is a geodesic lamination together with a transverse measure. We define the \textit{space of measured laminations} $\ML(\Sigma)$ to be the space of all homotopy classes of measured geodesic laminations on $\Sigma$ with compact support. These definitions don't depend on the hyperbolic structure chosen, but only on the topology of $\Sigma$; see, for example, \cite{pen_com}. Multiplying the transverse measure on a geodesic lamination by a positive constant gives an action of $\R_+$ on $\ML(\Sigma)$. We can therefore define the set of \textit{projective measured (geodesic) laminations} $\PML(\Sigma)$ on $\Sigma$ as the quotient $$\PML(\Sigma) = (\ML(\Sigma)\setminus 0)/\R_+,$$ where $0$ is the empty lamination.
\subsection{Complex projective structure} \label{sub:complex}
	
A \textit{(complex) projective structure} on a surface $\Sigma$ is a $(\PSL(2, \C), \hat\C)$--structure on $\Sigma$, consisting of a (maximal) open covering $\{U_{i}: i \in I\}$ of $\Sigma$, homeomorphisms $\Phi_{i}\co U_{i} \to V_i \subset \hat{\C}$ such that for all connected components $W$ of $U_{i} \cap U_{j}$, the transition functions $\Phi_{i} \circ \Phi_{j}^{-1}|_{\Phi_{j}(W)}$ are the restriction of some $g \in {\rm PSL}(2,\C)$. 

We can define the \textit{space of marked (complex) projective structure} $\P(\Sigma)$ as the set of equivalence classes $[(f,Z)]$ of pairs $(f,Z)$, where $Z$ is a (complex) projective structure on $\Sigma$ and $f \co \mathrm{int}(\Sigma) \to Z$ is an orientation preserving homeomorphism. Two pairs $(f_1,Z_1)$ and $(f_2,Z_2)$ are equivalent in $\P(\Sigma)$ if there is an orientation-preserving diffeomorphism $g \co Z_1 \to Z_2$ such that $g \circ f_1$ is isotopic to $f_2$. 

\subsubsection{Developing map and (groupoid) holonomy representation} \label{ssub:dev}

To every complex projective structure (and, more generally, to any $(G,X)$--structure) on a surface $\Sigma$, we can associate a pair $({\rm Dev},\rho)$, where:
\begin{itemize}
  \item $\rho$ is a homomorphism $\rho\co \pi_{1}(\Sigma) \to {\rm PSL}(2,\C)$, called the \textit{holonomy representation};
  \item ${\rm Dev}$ is an immersion ${\rm Dev}\co\widetilde\Sigma \to \hat{\C}$ from the universal covering space $\tilde{\Sigma}$ of $\Sigma$ to the Riemann sphere $\hat{\C}$, called the \textit{developing map}, equivariant with respect to $\rho$ and such that the restriction of $f$ to any sufficiently small open set in $\widetilde\Sigma$ is a projective chart for $Z$.
\end{itemize}

A projective structure on $\Sigma$ lifts to a projective structure on the universal cover $\widetilde{\Sigma}$. Then, a developing map can be constructed by analytic continuation starting from any base point $x_0$ in $\tilde{\Sigma}$ and any chart defined on a neighbourhood $U$ of $x_0$. Another chart (defined on $U'$) that overlaps $U$ can be modified by a M\"obius transformation so that it agrees on the overlap, in such a way that we can define a map from $U \cup U'$ to $\hat\C$. Continuing with this method one defines a map on successively larger subsets of $\tilde{\Sigma}$. The fact that $\widetilde\Sigma$ is simply connected is essential because nontrivial homotopy classes of loops in the surface create obstructions to this process. The holonomy representation $\rho \co \pi_1(\Sigma) \to \PSL(2,\C)$ is described as follows. A path $\gamma$ in $\Sigma$ passes through an ordered chain of simply connected open sets $U_{0}, \ldots, U_{n}$ such that $U_{i} \cap U_{i+1}$ is connected and non-empty for every $i =0, \ldots,n-1$. This defines the overlap maps $R_{i} = \Phi_{i} \circ \Phi_{i+1}^{-1}$ for  $i =0, \ldots, n-1$. The sets $V_i$ and $R_{i}(V_{i+1})$ overlap in $\hat \C$ and hence the developing image of $\widetilde \gamma$ in $\Chat$ passes through, in order, the sets $V_0, R_{0}(V_{1}), R_{0}R_1(V_{2})\ldots, R_{0}\cdots R_{n-1} (V_n)$. If $\gamma$ is closed, we can ask $U_{n} = U_{0}$ so that $V_0 = V_n$. Then, by definition, the holonomy of the homotopy class $[\gamma]$ is $\rho([\gamma]) = R_{0}\cdots R_{n-1} \in \PSL(2,\C)$. The group ${\rm PSL}(2,\C)$ acts on the sets of pairs $({\rm Dev},\rho)$ in the following way: given $A \in \PSL(2,\C)$, then we have $$A \cdot ({\rm Dev},\rho(\cdot)) = (A \circ{\rm Dev}, A\rho(\cdot)A^{-1}).$$

\subsubsection{Topology on $\P(\Sigma)$} 

We give $\P(\Sigma)$ the topology induced by uniform convergence of charts, or, equivalently, the quotient topology induced by the compact-open topology on the set of pairs $(\mathrm{Dev}, \rho)$. This topology is also equivalent to the locally uniform convergence of the developing maps. The space $\P(\Sigma)$ is a finite--dimensional complex manifold, diffeomorphic to a ball in $\R^{4\xi}$, where $\xi = \xi(\Sigma_{g, b}) = 3g-3+b$ is the complexity of the surface, see Dumas \cite{dum_com}.

\subsection{Teichm\"uller space and quasi-Fuchsian space} \label{sub:kleinian_groups} 

Given an oriented surface $\Sigma$ of negative Euler characteristic, the \textit{Teichm\"uller space} $\T(\Sigma)$ is the space of marked complex structures on $\Sigma$. Using the Uniformisation Theorem, the Teichm\"uller space can also be defined as the space of marked complete finite area hyperbolic structure on $\Sigma$. 

A \textit{Fuchsian group} is a discrete subgroup of $\PSL(2,\R)$, while a \textit{Kleinian group} $G$ is a discrete subgroup of $\PSL(2,\C)$. A Kleinian group $G$ acts by isometries on $\HH^3$ and by conformal automorphisms on the sphere at infinity $\hat{\C} = \C \cup \{\infty\}$. While the action of $G$ on $\HH^3$ is properly discontinuous, there are accumulation points on $\hat{\C}$. The limit set $\Lambda(G)$ is the closure of the set of accummulation points for the action of $G$ on $\HH^3$ (or $\hat{\C}$). The domain of discontinuity $\Omega(G)$ of $G$ is the set $\hat{\C} \setminus G$ and is the biggest domain of discontinuity for the action of $G$. Fuchsian groups are Kleinian groups. Another example of Kleinian groups are \textit{quasi-Fuchsian groups}, which are Kleinian group such that the limit set $\Lambda(G)$ is a topological circle. If $G\simeq \pi_1(\Sigma)$ is quasi-Fuchsian, then the associated $3$--manifold $M_G = \HH^3/G$ is homeomorphic to $\Sigma \times (-1,1)$, and $\Omega(G)$ has exactly two simply connected $G$--invariant components $\Omega^\pm$ such that the ``complex structures at infinity'' $\Omega^\pm / G$ are homeomorphic to $\Sigma$. The space of marked groups $G \simeq \pi_1(\Sigma)$ such that $G$ is Fuchsian (up to conjugation) is called \textit{Fuchsian space} $\F(\Sigma)$, while the space of marked groups $G \simeq \pi_1(\Sigma)$ such that $G$ is quasi-Fuchsian (up to conjugation) is called \textit{quasi-Fuchsian space} $\QF(\Sigma)$. The Teichm\"uller space $\T(\Sigma)$ can be identified with the Fuchsian space $\F(\Sigma)$, while Bers' Simultaneous Uniformization Theorem says that $\QF(\Sigma)$ can be parametrized by the pair of complex structures at infinity. 

\subsection{Three manifolds and pleating rays} \label{ssub:manifolds}

Let $M$ be a hyperbolic $3$--manifold. An important subset of $M$ is its \emph{convex core} $\CC_{M} = \CC$, that is the smallest, non-empty, closed, convex subset of $M$ such that the inclusion of $\CC_{M}$ into $M$ is a homotopy equivalence. Thurston proved that, if $M$ is geometrically finite, then there is a natural homeomorphism between the components of $\partial\CC_{M}$ and the components of $\Omega/G$. Thurston proved that each such component $F$ is a (locally convex) embedded pleated surface, that is, it is a hyperbolic surface which is totally geodesic almost everywhere and such that the locus of points where it fails to be totally geodesic is a geodesic lamination, called \textit{bending (or pleated) lamination}. Since the pleated surface is locally convex, the lamination carries a natural transverse measure, called the \textit{bending measure} (or \textit{pleating measure}), see~\cite{eps_con}.

\subsection{Fenchel-Nielsen coordinates for $\F(\Sigma)$ and $\QF(\Sigma)$} \label{sub:fenchel_nielsen_coordinates}

The Fenchel--Nielsen maps $\FN_{\R}$ and $\FN_{\C}$ provide a parametrization of the Fuchsian space $\F(\Sigma)$ and of the quasi-Fuchsian space $\QF(\Sigma)$, respectively. They are defined with respect to a marking decomposition $(\PC, \DD)$ of $\Sigma$ consisting of a pants decomposition $\PC = \{\s_1, \ldots, \s_\xi\}$ and a marking decomposition $\DD = \{D_1, \ldots, D_\xi\}$. 

For the map $\FN_{\R}$, there are two types of coordinates:
\begin{itemize}
  \item the \textit{length parameters} $l_i \in \R_{+}$, which measure the length of the pants curves $\s_i$ in $\HH^2/G$, that is $l_i = l_{G}(\s_i)$ is the length of the curve $\s_i$ in the hyperbolic surface $\HH^2/G$.
  \item the \textit{twist parameters} $t_i \in \R$ which measure the relative positions along $\s_i \in \PC$ in which the pants $P$ and $P'$ (adjacent to $\s_i$ and not necessarily distinct) are glued to form $\HH^2/G$. 
\end{itemize}

The choice of the marking $(\PC, \DD)$ on $\Sigma$ lets us distinguish the effects of the Dehn twists about the pants curves $\s_i$. Fenchel--Nielsen theorem states that the map $$\FN_{\R} \co \F \to \R_{+}^\xi \times \R^\xi$$ defined by $$\FN_{\R}(G) = (\frac{l_1}{2},\ldots, \frac{l_\xi}{2}, t_1,\ldots t_\xi) = (\frac{\ull}{2}, \utt)$$ is a real analytic bijection. We define the length parameters as half-lengths $\frac{l_i}{2}$ to be consistent with the definition of the complex Fenchel-Nielsen coordinates below.

Tan~\cite{tan_com} and Kourouniotis~\cite{kou_com} showed that the Fenchel--Nielsen parametrization can be extended to the quasi-Fuchsian space by `complexifying' the parameters. In particular they replaced the length coordinates $l_i$ with the \textit{complex (translation) length} coordinates $\l_i$ of the element $S_i$ representing the curve $\s_i$, defined by the formula $\Tr (S_i) = 2 \cosh\frac{\l_i}{2}$ and chosen so that $\Re \l_i > 0$. From the periodicity of $\cosh$, we have that $\frac{\l_i}{2} \in \C_+/2\pi i\mathbb{Z},$ where $\C_{+} = \{z \in \C | \Re z > 0\}$.  There is an ambiguity of sign which depends on whether one chooses the half--length as $\frac{\l_i}{2}$ or $\frac{\l_i}{2}+i\pi$, so one needs to specify the (complex) half--lengths $\frac{\l_i}{2}$, rather than the (complex) lengths $\l_i$. The \textit{complex twist parameter} $\t_i$ describes how to glue together two neighbouring pants. With suitable conventions, it is the signed complex distance between the oriented common perpendiculars to lifts of appropriate boundary curves in the two pairs of pants, measured along their oriented common axis.

Tan~\cite{tan_com} and Kourouniotis~\cite{kou_com} proved that the map $$\FN_{\C} \co \QF \to (\C_{+}/2\pi i)^\xi \times (\C/2\pi i)^\xi$$ defined by $$\FN_{\C}(G) = (\frac{\l_1}{2},\ldots, \frac{\l_\xi}{2}, \t_1,\ldots \t_\xi) = (\frac{\ul}{2}, \ut)\in (\C_{+}/2i\pi)^\xi \times (\C/2\pi i)^\xi$$ is a holomorphic embedding. This map, when restricted to $\F(\Sigma)$, coincides with the map $\FN_{\R}$ defined above, see Theorem 1 of Tan~\cite{tan_com}. 

\section{The c--gluing construction} \label{sub:c_glu}

In this section we are going to describe the $\uc$--\textit{gluing construction}. Let $\Sigma = \Sigma_{g, b}$ be an hyperbolizable orientable surface of genus $g$ and with $b$ punctures and let $\xi = \xi(\Sigma) = 3g-3+b$ be its complexity. Given a pants decomposition $\PC = \{\s_1, \ldots, \s_\xi\}$ on $\Sigma$ and a vector $\uc = (c_1, \ldots, c_\xi) \in \R_{+}^\xi$, the  $\uc$--gluing construction defines a complex projective structure on $\Sigma$ so that the holonomies of all the loops $\s_j \in \PC$ are hyperbolic elements of $\PSL(2, \C)$ with translation length $2c_j$. The idea is based on the gluing construction that the author and Series described in \cite{mal_top} by reinterpreting Kra's plumbing construction \cite{kra_hor}. More precisely, first we fix an identification of the interior of each pair of pants $P_i$ to a standard three-holed sphere endowed with the projective structure coming from the unique hyperbolic metric on a three holed sphere with fixed boundary lengths (defined by $\uc$). Then, we glue, or ``plumb'', adjacent pants by deleting open neighbourhoods of the two ends in question and gluing the two pairs of pants along horocyclic annular collars around the two boundary curves. The gluing across the $i$-th pants curve is defined by a parameter $\boldsymbol\mu_j \in \C_{[0,\pi)} = \{z \in \C \mid \mathrm{Im}(z) \in [0, \pi)\}$ for $j = 1, \ldots, \xi$. This defines a complex projective structure with developing map $\Dev_{\uc, \umb}\co \widetilde{\Sigma} \to \mathbb{CP}^1$ and holonomy $\rho_{\uc, \umb}\co \pi_1(\Sigma) \to \PSL(2,\C)$, where $\umb = (\boldsymbol\mu_1, \ldots, \boldsymbol\mu_{\xi})\in (\C_{[0,\pi)})^{\xi}$. We refer to this `new' gluing construction as the $\uc$--\textit{gluing construction}. 

Our construction is similar to the construction described in \cite{kra_hor, mal_top}, but we replaced the thrice-punctured sphere with three-holed sphere. As one might expect, when $\uc$ tends to $\underline{0} = (0, \ldots, 0)$, this generalised construction limits to the gluing construction defined in \cite{mal_top}, as we will prove in Section \ref{subsec:lim}. (The convention of using bold letters will be more clear in that section, when we will take limits as $\uc \to \underline{0}$.) 

\subsection{The standard three-holed sphere} \label{sec:std}

The first step for describing the gluing is to define the structure on the `standard' pairs of pants. Any three holed sphere whose boundary components have length $2c_{1}, 2c_{2}$ and $2c_{3}$, where $c_i\in \R_{\ge 0}$, is isometric to a {\em standard pair of pants} $\mathbf{P}(c_{1}, c_{2}, c_{3})$, which can be defined as $$\mathbf{P}(c_{1}, c_{2}, c_{3}) := \HH^2/\boldsymbol\Gamma(c_{1}, c_{2}, c_{3}),$$ where $\boldsymbol\Gamma(c_{1}, c_{2}, c_{3}) = \langle A_\infty, A_0, A_1 |  A_\infty A_0 A_1 = \Id  \rangle$, and
$$A_\infty= A_\infty(c_{1},c_{2},c_{3}) = \left(
\begin{array}{cc}
\cosh c_{1}  & \cosh c_{1}+1 \\
\cosh c_{1}-1 & \cosh c_{1} \\
\end{array}
\right),$$ 
$$A_0 = A_0(c_{1},c_{2},c_{3}) =  \left(
\begin{array}{cc}
\cosh c_{2} & -\coth(\frac{c_{1}}{2})\tanh(\frac{\nu_1}{2}) \sinh c_{2} \\
-\tanh(\frac{c_{1}}{2})\coth(\frac{\nu_1}{2}) \sinh c_{2} & \cosh c_{2} \\
\end{array}
\right),$$ 
$$A_1= A_1(c_{1},c_{2},c_{3}) =  \left(
\begin{array}{cc}
\cosh c_{3} - \frac{\sinh c_{1} \sinh c_{3}}{\sinh \nu_2} & \frac{\coth(\frac{c_{1}}{2}) \sinh c_{3} (\cosh c_{1} - \cosh \nu_2)}{\sinh \nu_2}  \\
-\frac{\tanh(\frac{c_{1}}{2}) \sinh c_{3} (\cosh c_{1} + \cosh \nu_2)}{\sinh \nu_2} & \cosh c_{3}+ \frac{\sinh c_{1} \sinh c_{3}}{\sinh \nu_2}\\
\end{array}
\right),$$ 
\begin{eqnarray}\label{nu}
  \coth \nu_1 = \frac{\cosh c_{1}\cosh c_{2}+\cosh c_{3}}{\sinh c_{1}\sinh c_{2}},\;\;\coth\nu_2 = \frac{\cosh c_{1}\cosh c_{3}+\cosh c_{2}}{\sinh c_{1}\sinh c_{3}},
\end{eqnarray}
and $\nu_1, \nu_2 > 0.$ In this section we describe the calculations when $c_i > 0$ and the other cases are discussed in Appendix \ref{app:degenerate}.

The fixed points of the elements $A_\infty, A_0$ and $A_1$ are the following:
\begin{itemize}
  \item $\Fix^\pm(A_\infty) = \{\pm \coth(\frac{c_{1}}{2}) \}$;
  \item $\Fix^\pm(A_0) = \{\pm \coth(\frac{c_{1}}{2})\tanh(\frac{\nu_1}{2}) \}$;
  \item $\Fix^\pm(A_1) = \{\coth(\frac{c_{1}}{2})\tanh(\frac{c_{1}-\nu_2}{2}), \coth(\frac{c_{1}}{2})\tanh(\frac{c_{1}+\nu_2}{2}) \}$.
\end{itemize}

These calculations are inspired by Maskit \cite[Section 5.2]{mas_mat}, and by Parker and Parkkonen \cite{par_coo} regarding the choice of the conjugation class. The choice of the normalisation is done so that there exists a limit when $\uc \to \underline{0}$. In fact, in order for the limit to exist, we need to choose well the conjugation class for the subgroup of $\PSL(2,\C)$. In Appendix \ref{app:degenerate} we explain in more details this choice, and the calculations in the case that one or more of the boundary components have length zero (which is necessary to do if $\Sigma$ has punctures). The discussion below can be generalized to these cases as well. 

\begin{figure}
\centering
\includegraphics[height=5cm]{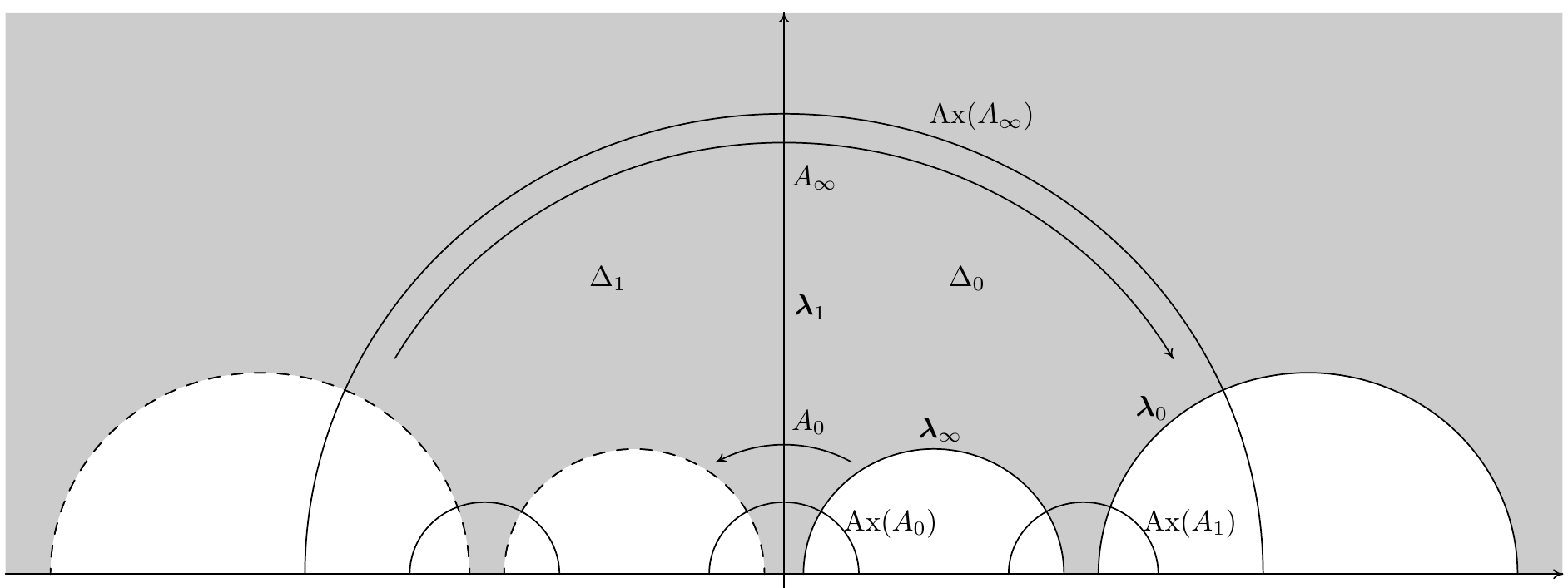}
\caption{The fundamental set ${\boldsymbol\Delta}(c_{1},c_{2},c_{3})$ for $\boldsymbol\Gamma(c_{1},c_{2},c_{3})$, where ${\boldsymbol\lambda}_\e$ is the geodesic between $\Ax (A_{\e+1})$ and $\Ax (A_{\e+2})$.}
\label{fig9}
\end{figure}

Let's fix a standard fundamental set ${\boldsymbol\Delta}$  for the action of $\boldsymbol\Gamma(c_{1}, c_{2}, c_{3})$ on $\HH = \HH^2$:
  $${\boldsymbol\Delta} = {\boldsymbol\Delta}(c_{1},c_{2},c_{3}) = \{z \in \HH^{2} | |z - C_1|\geq r_1, |z - C_2|\geq r_2, |z + C_1| \geq r_1, |z + C_2| \geq r_2 \},$$ 
where:
\begin{itemize}
  \item $C_1 = \frac{\cosh c_{1}}{\cosh c_{1} -1}$, $r_1 = \frac{1}{\cosh c_{1} -1}$;
  \item $C_2 = -\tanh(\frac{c_{1}}{2})\coth(\frac{\nu_1}{2}) \tanh c_{2}$ and $r_2 = \tanh(\frac{c_{1}}{2})\coth(\frac{\nu_1}{2}) \sinh c_{2}$.
\end{itemize}
See Figure \ref{fig9}. The circles with centers $\pm C_1$ and radius $r_1$ are the isometric circles of $A_\infty$ and $A_\infty^{-1}$, while the circles with centers $\pm C_2$ and radius $r_2$ are the isometric circles of $A_0^{-1}$ and $A_0$. Let $$\boldsymbol\pi = \boldsymbol\pi_{c_{1},c_{2},c_{3}}\co \HH \to \mathbf{P}(c_{1},c_{2},c_{3})$$ be the natural quotient map. 

The axes $\Ax(A_\infty)$, $\Ax(A_0)$, $\Ax(A_1)$ and $\Ax(A_0A_\infty)$ projects under $\boldsymbol\pi$ to the three closed boundary geodesics in $\mathbf{P}(c_{1},c_{2},c_{3})$, while the images of the geodesics 
\begin{itemize}
  \item ${\boldsymbol\lambda}_0 = \{z \in \HH^{2} | |z - C_1| = R_1\}$,
  \item ${\boldsymbol\lambda}_\infty =\{z \in \HH^{2} | |z - C_2| = R_2 \}$, 
  \item ${\boldsymbol\lambda}_1 = \{z \in \HH^{2} | \Re z = 0 \}$
\end{itemize}
under $\boldsymbol\pi$ correspond to the seams of $\mathbf{P}(c_{1},c_{2},c_{3})$ which intersect the boundary geodesics orthogonally. These seams split $\mathbf{P}(c_{1},c_{2},c_{3})$ into two (infinite area) `hexagons', which correspond to the subsets ${\boldsymbol\Delta}_0$ and ${\boldsymbol\Delta}_1$ defined by
\begin{itemize}
  \item ${\boldsymbol\Delta}_0 = {\boldsymbol\Delta}_0(c_{i_1}, c_{i_2}, c_{i_3}) = {\boldsymbol\Delta} \cap \HH_{\ge 0}$, where $\HH_{\ge 0} = \{z \in \HH \mid \Re z \geq 0\}$. 
  \item ${\boldsymbol\Delta}_1 = {\boldsymbol\Delta}_1(c_{i_1}, c_{i_2}, c_{i_3}) = {\boldsymbol\Delta} \cap \HH_{\le 0}$, where $\HH_{\le 0} = \{z \in \HH \mid \Re z \leq 0\}$. 
\end{itemize} 
We will refer to ${\boldsymbol\Delta}_0$ and ${\boldsymbol\Delta}_1$ as the \textit{white} and the \textit{black} regions, respectively.

\subsection{The c--gluing} \label{sec:cglu}

Fix $\uc\in \R_{>0}^\xi$. The pants decomposition $\PC$ determines the set $\P = \{P_1, \ldots, P_k\}$ of pairs of pants of $\Sigma = \Sigma_{g, b}$, where $k = -\chi(\Sigma) = 2g-2+b$. Any pair of pants $P_j$ has three boundary components $\s_{i_1}, \s_{i_2}$ and $\s_{i_3}$ which could be pants curves in $\PC$ or punctures of $\Sigma$. If $\s_{i_k} \in \PC$, let $c_{i_k}$ be the positive real number fixed by $\uc$, while if $\s_{i_k}$ corresponds to a puncture of $\Sigma$, let $c_{i_k} = 0$. 

For every $P_j \in \P$ with boundary components $\s_{i_1}, \s_{i_2}$ and $\s_{i_3}$, fix an homeomorphism $${\boldsymbol\Phi}_j \co \mathrm{Int}(P_j) \to \mathbf{P}(c_{i_1}, c_{i_2}, c_{i_3})$$ from the interior of the pair of pants $P_j$ to the standard pair of pants $\mathbf{P} = \mathbf{P}(c_{i_1},c_{i_2},c_{i_3})$. This identifications induce a labelling of the three boundary components of $P_{j}$ as $\partial_0 P_{j}, \partial_1 P_{j}, \partial_{\infty} P_{j}$ in some order, fixed from now on, and a coloring of the two regions whose union is $P_j$, one being `white' and one being `black'. We denote the geodesic boundary curves of $\mathbf{P}(c_{i_1}, c_{i_2}, c_{i_3})$ of length $2c_{i_1}$, $2c_{i_2}$ and $2c_{i_3}$, respectively, as $\dd_\infty\mathbf{P}, \dd_0\mathbf{P}$ and $\dd_1\mathbf{P}$. Suppose that the pairs of pants $P$ and $P'$ in $\P$ are adjacent along the pants curve $\s = \s_{i_1}$ (of length $c_{i_1}$) corresponding to the boundaries $\dd_{\e}P$ and $\dd_{\e'}P'$. (If $P =P'$ then clearly $\e \neq \e'$.) The gluing across $\s$ is described by a complex parameter $\boldsymbol\mu =  {\boldsymbol\mu}_{i_1}\in \C_{[0,\pi)} = \{z \in \C \mid \mathrm{Im}(z) \in [0, \pi)\}$.  

\begin{figure}
\centering 
\includegraphics[height=14cm]{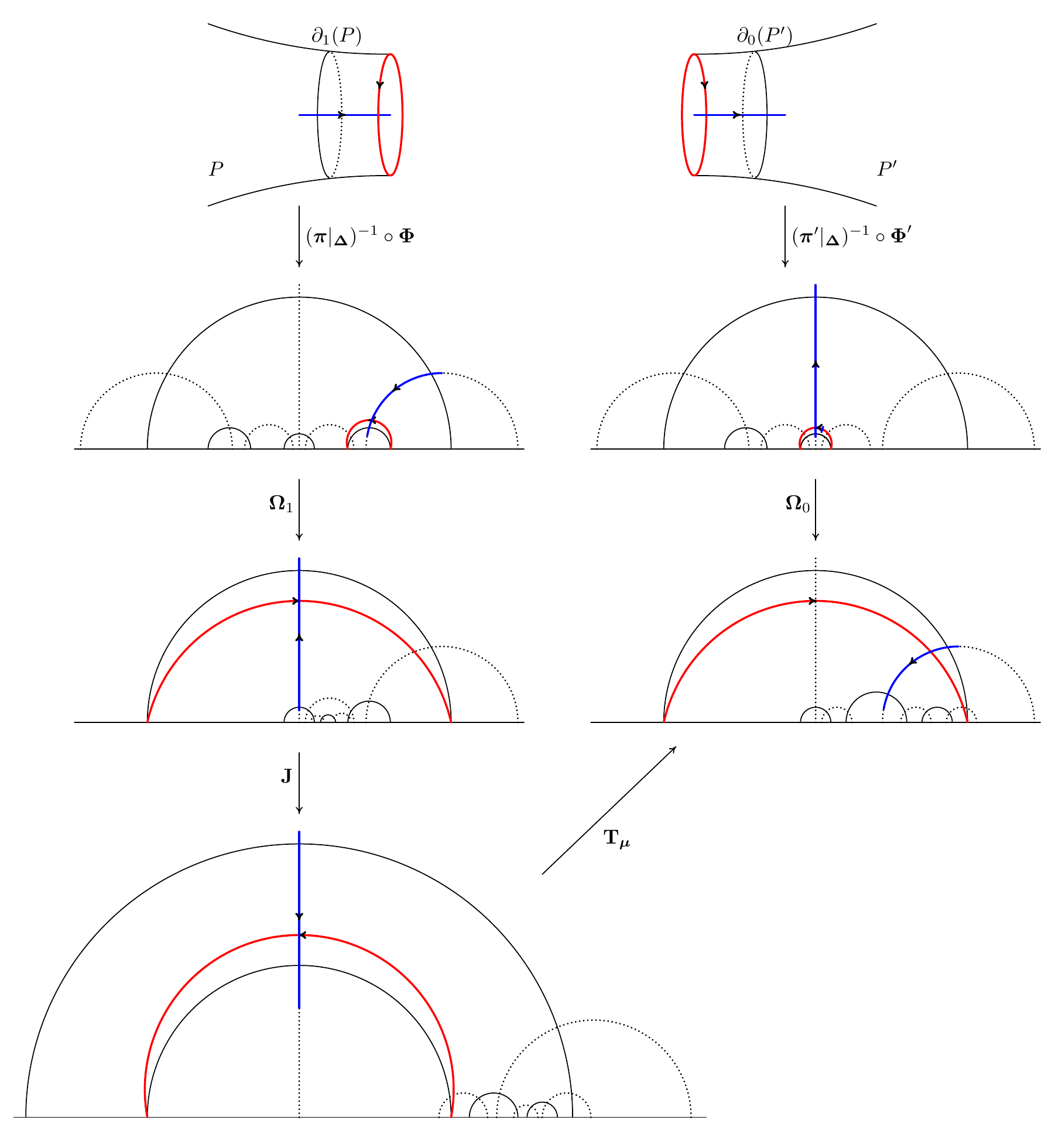} 
\caption{The gluing construction when $\e = 1$ and $\e' = 0$.} 
\label{fig10}
\end{figure}

We first describe the gluing in the case $\e = \e' = \infty$. Let $\mathbf{P} = \mathbf{P}(c_{i_1}, c_{i_2}, c_{i_3})$ and $\mathbf{P}'= \mathbf{P}(c_{i_1},c_{i_4},c_{i_5})$, where $c_{i_j}$ are determined as described above.  We refer to the copy of $\HH$ associated to $\mathbf{P}'$ as $\HH'$. We define the projections $$\boldsymbol\pi \co \HH \to \mathbf{P}(c_{i_1}, c_{i_2}, c_{i_3})\;\;\text{ and }  \;\; \boldsymbol\pi' \co \HH' \to \mathbf{P}'(c_{i_1},c_{i_4},c_{i_5}).$$ Arrange the pairs of pants with $P$ on the left as shown in Figure~\ref{fig10}. (Note that the figure describes the more general case $\e = 1$ and $\e' = 0$.) In Figure~\ref{fig10} the top two arrows corresponds to the maps $$({\boldsymbol\pi}|_{\boldsymbol\Delta})^{-1} \circ \boldsymbol\Phi  \co \mathrm{Int}(P) \to {\boldsymbol\Delta}(c_{i_1},c_{i_2},c_{i_3}) \;\;\;\; \text{and}\;\;\;\;({\boldsymbol\pi'}|_{\boldsymbol\Delta})^{-1} \circ \boldsymbol\Phi' \co \mathrm{Int}(P') \to {\boldsymbol\Delta}(c_{i_1},c_{i_4},c_{i_5}).$$

\vspace{.15 in}

Recall that, given a geodesic $\g$ in $\HH$, an \textit{hypercycle} around $\g$ consists of the points at constant distance from $\g$. For example, if $\g$ is the geodesic between $0$ and $+\infty$, then the hypercycles around $\g$ are defined by $$\{z \in \HH | \Arg(z) = \frac{\pi}{2}-\th\},$$ where $\th \in (-\frac{\pi}{2},\frac{\pi}{2})$ is fixed. On the other hand, if $\g$ is the geodesic between $-r$ and $r$ (with $r \in \R_+$), then the hypercycles around $\g$ are defined by $$\{z \in \HH : |z+ir\tan(\th)| = r\mathrm{sec}(\th)\},$$ where $\th \in (-\frac{\pi}{2},\frac{\pi}{2})$ is fixed.

 \vspace{.15in}

Let $\mathbf{h}_\infty = \mathbf{h}_{\infty}(c_{i_1}, c_{i_2}, c_{i_3})$ be the projection to $\mathbf{P}(c_{i_1}, c_{i_2}, c_{i_3})$ of the `inner' hypercycle $$\mathbf{h}_{\infty, \HH} =  \mathbf{h}_{\infty, \HH}(c_{i_1}, c_{i_2}, c_{i_3}; {\boldsymbol\mu}_{i_1}) = \{z \in \HH : |z+ir\tan(\th)| = r\mathrm{sec}(\th)\}$$ around $\Ax(A_\infty(c_{i_1}, c_{i_2}, c_{i_3}))$, where $\th = \frac{\Im {\boldsymbol\mu}_{i_1}}{2} \in [0,\frac{\pi}{2})$ and $r = \coth(\frac{c_{i_1}}{2})$. The choice $\th = \frac{\Im {\boldsymbol\mu}_{i_1}}{2}$ will be more clear after the discussion in Section \ref{sec:lim}. For $\nu>0$, the region $$\mathbf{H}_{\infty}(c_{i_1},c_{i_2},c_{i_3}; {\boldsymbol\mu}_{i_1}, \nu) = \{z \in \HH\; |\; d_{\HH}(z, \mathbf{h}_{\infty}) < \nu\} \subset \HH$$ projects, under $\boldsymbol\pi$, to an (open) annulus $$\mathbf{A}_\infty = \mathbf{A}_\infty(c_{i_1}, c_{i_2}, c_{i_3}) = \boldsymbol\pi\left(\mathbf{H}_{\infty}(c_{i_1},c_{i_2},c_{i_3}; {\boldsymbol\mu}_{i_1}, \nu)\right)$$ around $\mathbf{h}_{\infty} = \boldsymbol\pi(\mathbf{h}_\infty)$ (and parallel to $\partial_{\infty}\mathbf{P}$). The outer boundary $\partial^+(\mathbf{A}_\infty)$ of $\mathbf{A}_\infty$ bounds a closed infinite area annulus on $\mathbf{P}$. Let $\mathbf{S}$ be the surface obtained by removing this closed infinite area annulus from $\mathbf{P}$. Note that $\mathbf{S}$ is open. Define $\mathbf{h}'_{\infty}(c_{i_1}, c_{i_4}, c_{i_5})$, $\mathbf{h}'_{\infty, \HH}(c_{i_1}, c_{i_4}, c_{i_5})$, $\mathbf{H}'_{\infty}(c_{i_1}, c_{i_4}, c_{i_5})$, $\mathbf{A}'_\infty = \mathbf{A}'_\infty(c_{i_1}, c_{i_4}, c_{i_5})$ and $\mathbf{S}'$ in a similar way. 

We are going to glue the truncated surfaces $\mathbf{S}$ to $\mathbf{S}'$ by matching $\mathbf{A}_\infty$ to $\mathbf{A}'_\infty$ in such a way that $\mathbf{h}_{\infty}$ is identified to $\mathbf{h}'_{\infty}$,  $\partial^+(\mathbf{A}_\infty)$ is identified with $\partial^-(\mathbf{A}'_\infty)$ and $\partial^-(\mathbf{A}_\infty)$ is identified with $\partial^+(\mathbf{A}'_\infty)$, see Figures~\ref{fig10}. The resulting homotopy class of the loop $\mathbf{h}_{\infty}$ on the glued up surface (the quotient of the disjoint union of the surfaces $\mathbf{S}$ and $\mathbf{S}'$ by the gluing map mentioned above) will be in the homotopy class of $\s$. To keep track of the marking on $\Sigma$, we do the gluing at the level of the $\Z$--covers of $\mathbf{A}_\infty$ and $\mathbf{A}'_\infty$ corresponding to $\mathbf{H}_{\infty}$ and $\mathbf{H}'_{\infty}$, that is, we actually glue the strips $\mathbf{H}_{\infty}$ and $\mathbf{H}'_{\infty}$. We will explain the details of this construction in Sections \ref{sec:projstructure} and \ref{sec:mark}.

As shown in Figure~\ref{fig10}, we first need to reverse the direction in one of the two strips $\mathbf{H}_{\infty}, \mathbf{H}'_{\infty}$. Set
\begin{equation}\label{eqn:standardsymmetries1_bold}
\begin{aligned}
\mathbf{J} &= \mathbf{J}(c_{i_1}, c_{i_2}, c_{i_3}) = \begin{pmatrix}
	0  &  -\coth(\frac{c_{i_1}}{2})  \\
	\tanh(\frac{c_{i_1}}{2}) &  0   
	\end{pmatrix} \text{ and }\\ 
\mathbf{T}_{\boldsymbol\mu}&= \mathbf{T}_{\boldsymbol\mu}(c_{i_1}, c_{i_2}, c_{i_3}) = \begin{pmatrix}
\cosh \frac{\boldsymbol\mu}{2} & -\sinh\frac{\boldsymbol\mu}{2}\coth \frac{c_{i_1}}{2} \\
-\sinh\frac{\boldsymbol\mu}{2}\tanh \frac{c_{i_1}}{2} & \cosh\frac{\boldsymbol\mu}{2} \notag
\end{pmatrix}.
\end{aligned}
\end{equation}
We reverse the direction in $\mathbf{H}_{\infty}$ by applying the map $\mathbf{J}(z) = -\coth^2(\frac{c_{i_1}}{2}) \frac{1}{z}$, which corresponds to the rotation of angle $\pi$ around the point $\left(\coth \frac{c_{i_1}}{2}\right)i$. We then glue $\mathbf{H}_{\infty}$ to $\mathbf{H}'_{\infty}$ by identifying $z \in \mathbf{H}_{\infty}$ to $z' = \mathbf{T}_{\boldsymbol\mu} \mathbf{J}(z) \in \mathbf{H}'_{\infty}$. This identification descends to a well defined identification of $\mathbf{A}_\infty$ with $\mathbf{A}'_\infty$, in which the `outer' boundary of $\mathbf{A}_\infty$ is identified to the `inner' boundary of $\mathbf{A}'_\infty$. In particular, applying $\mathbf{T}_{\boldsymbol\mu} \mathbf{J}$, we glue $\mathbf{h}_{\infty}$ to $\mathbf{h}'_{\infty}$ reversing the orientation in $\HH^2$ (but not in $\HH^3$), as we wanted. Note that the map $\mathbf{T}_{\boldsymbol\mu} \mathbf{J}$ coincides with the map $U_{\boldsymbol\mu}^{-1}$, where $U = U_{\boldsymbol\mu}$ is defined in Section \ref{ssub:four}. Looking at the action of $\mathbf{T}_{\boldsymbol\mu} \mathbf{J}$ on $\HH \subset \hat\C$, we can see that this map sends the hypercycle $\mathbf{h}_{\infty, \HH}$ to itself. 

Now we discuss the general case in which $P$ and $P'$ meet along the boundary components $\partial_\e(P)$ and $\partial_{\e'}(P')$, where $\e, \e' \in \{0,1,\infty\}$. As above, let ${\boldsymbol\Delta}_0 = {\boldsymbol\Delta}_0(c_{i_1}, c_{i_2}, c_{i_3}) \subset \HH_{+}$ be the white `hexagon' of ${\boldsymbol\Delta}(c_{i_1}, c_{i_2}, c_{i_3})$. Notice that there is a unique orientation preserving map ${\boldsymbol\Omega}_{0} = {\boldsymbol\Omega}_{0}(c_{i_1}, c_{i_2}, c_{i_3}) \in \PSL(2, \C)$ such that
$${\boldsymbol\Omega}_{0} \left({\boldsymbol\Delta}_0(c_{i_3}, c_{i_1}, c_{i_2})\right) = {\boldsymbol\Delta}_0(c_{i_1}, c_{i_2}, c_{i_3})$$ and such that $\Ax\left(A_0(c_{i_3}, c_{i_1}, c_{i_2})\right) \subset {\boldsymbol\Delta}(c_{i_3}, c_{i_1}, c_{i_2})$ is mapped to $\Ax\left(A_\infty(c_{i_1}, c_{i_2}, c_{i_3})\right) \subset {\boldsymbol\Delta}(c_{i_1}, c_{i_2}, c_{i_3})$. Similarly, let ${\boldsymbol\Omega}_{1} = {\boldsymbol\Omega}_{1}(c_{i_1}, c_{i_2}, c_{i_3}) \in \PSL(2,\C)$ be the unique orientation preserving transformation such that 
 $${\boldsymbol\Omega}_{1}\left({\boldsymbol\Delta}_0(c_{i_2}, c_{i_3}, c_{i_1}) \right) = {\boldsymbol\Delta}_0(c_{i_1}, c_{i_2}, c_{i_3})$$ and such that $\Ax\left(A_1(c_{i_2}, c_{i_3}, c_{i_1})\right)$ in ${\boldsymbol\Delta}(c_{i_2}, c_{i_3}, c_{i_1})$ is mapped to $\Ax\left(A_\infty(c_{i_1}, c_{i_2}, c_{i_3})\right)$ in ${\boldsymbol\Delta}(c_{i_1}, c_{i_2}, c_{i_3})$. To do the gluing, first move $\partial_\e\mathbf{P}(c_{i_3}, c_{i_1}, c_{i_2})$  to $\partial_\infty\mathbf{P}(c_{i_1}, c_{i_2}, c_{i_3})$ and $\partial_{\e'}\mathbf{P}'(c_{i_5}, c_{i_1}, c_{i_4})$ to $\partial_\infty\mathbf{P}'(c_{i_1}, c_{i_4}, c_{i_5})$ using the maps $\boldsymbol\Omega_{\e}$ and $\boldsymbol\Omega_{\e'}$, respectively, and then proceed as before. Let 
\begin{equation}
\begin{aligned}
  \mathbf{h}_{0, \HH} &= \mathbf{h}_{0, \HH}(c_{i_3}, c_{i_1}, c_{i_2}; {\boldsymbol\mu}_{i_1}) = \boldsymbol\Omega_{0}^{-1}(\mathbf{h}_{\infty, \HH}),\\
  \mathbf{H}_0 &= \mathbf{H}_0(c_{i_3}, c_{i_1}, c_{i_2};{\boldsymbol\mu}_{i_1}, \nu) = \boldsymbol\Omega_{0}^{-1}(\mathbf{H}_\infty),\\
  \mathbf{h}_{1, \HH} &= \mathbf{h}_{1, \HH}(c_{i_2}, c_{i_3}, c_{i_1}; {\boldsymbol\mu}_{i_1}) = \boldsymbol\Omega_{1}^{-1}(\mathbf{h}_{\infty, \HH}),\\
  \mathbf{H}_1 &= \mathbf{H}_1(c_{i_2}, c_{i_3}, c_{i_1}; {\boldsymbol\mu}_{i_1}, \nu) = \boldsymbol\Omega_{1}^{-1}(\mathbf{H}_\infty),
\end{aligned}
\end{equation}
where $\mathbf{h}_{\infty, \HH} = \mathbf{h}_{\infty, \HH}(c_{i_1}, c_{i_2}, c_{i_3}; {\boldsymbol\mu}_{i_1})$ and $\mathbf{H}_\infty = \mathbf{H}_\infty(c_{i_1}, c_{i_2}, c_{i_3}; {\boldsymbol\mu}_{i_1}, \nu)$. Let also $\mathbf{h}_\e$ and $\mathbf{A}_\e$ be the projections (under $\boldsymbol\pi$) to $\mathbf{P}(c_{i_1}, c_{i_2}, c_{i_3})$ of the set $\mathbf{h}_{\e, \HH}$ and $\mathbf{H}_\e$, respectively. Thus the gluing identifies $z \in \mathbf{H}_\e$ to $z' \in \mathbf{H}_{\e'}$ by the formula 
\begin{equation}\label{eqn:gluing_bold} 
  \boldsymbol\Omega_{\e'}(z')= \mathbf{T}_{\boldsymbol\mu} \circ \mathbf{J} \left(\boldsymbol\Omega_{\e}(z)\right),
\end{equation}
see Figure~\ref{fig10}.

Finally, we carry out the above construction for each pants curve $\s_i \in \PC$ using gluing parameters $\umb = (\boldsymbol\mu_1, \ldots, \boldsymbol\mu_\xi) \in (\C_{[0,\pi)})^\xi$. To do this, we need to ensure that the annuli $\mathbf{H}_0$, $\mathbf{H}_1$ and $\mathbf{H}_\infty$ corresponding to the three different boundary components of a given pair of pants $P_j$ are disjoint. (Note that this is similar to the gluing construction of \cite{mal_top} where we needed to ask that the three horocycles were disjoint.)  Under this condition we can choose $\nu>0$ so that $\mathbf{H}_0$, $\mathbf{H}_1$ and $\mathbf{H}_\infty$ are disjoint in $\boldsymbol\Delta(c_{i_1}, c_{i_2}, c_{i_3})$, as required. We define $$\mathbf{S}(\uc, \umb) := \mathbf{S}_1 \sqcup \ldots \sqcup \mathbf{S}_k/\sim$$ to be the quotient of the disjoint union of the truncated surfaces $\mathbf{S}_j \subset \mathbf{P}(c_{i_1}, c_{i_2}, c_{i_3})$ defined above by the equivalence relation $\sim$ given by the attaching map along the annuli $\mathbf{A}_\e(\s_i)$ around each pants curve $\s_i$. Note that $\mathbf{S}(\uc, \umb)$ is homeomorphic to $\Sigma$.

As explained in details in the next section, this process defines a complex projective structure $\Sigma(\uc, \umb)$ on $\mathbf{S}(\uc, \umb)\cong \Sigma$.

\subsection{Projective structure}\label{sec:projstructure}

The $\uc$--gluing construction described in the previous section defines a marked complex projective structure on $\mathbf{S}(\uc, \umb)$ (and hence on $\Sigma$, since $\mathbf{S}(\uc, \umb)$ is homeomorphic to $\Sigma$). We describe the projective structure in this section and we discuss the marking in Section \ref{sec:mark}. The idea is the following. First, we define a complex projective structure on each truncated surface $\mathbf{S}_j = \mathbf{S}_j(\uc, \umb) \subset \mathbb P_j = {\boldsymbol\Phi}_j \left(\mathrm{Int}(P_j)\right)$, where $j = 1, \ldots, k$, and then, we describe why the attaching maps allow us to define a complex projective structure on the quotient $\mathbf{S}(\uc, \umb) = \mathbf{S}_1 \sqcup \ldots \sqcup \mathbf{S}_k/\sim$.

We recall some basic facts about complex projective structures that we will need later. See, for example, Dumas \cite{dum_com} for more details.
\begin{enumerate}
  \item Let $\Sigma' \subset \Sigma$ be an open subset and let $Z$ be a complex projective structure on $\Sigma$, then the restriction of $Z$ to $\Sigma'$ defines a complex projective structure on $\Sigma'$.
  \item A Fuchsian group $\Gamma \subset \PSL(2, \R)$ defines a projective structure on the quotient surface $\HH/\Gamma$.
\end{enumerate} 

These facts explain how to define a complex projective structure on each $\mathbf{S}_j \subset \mathbb P_j = \HH/\Gamma_j$ for $j = 1, \ldots, k$, where $\Gamma_j = \Gamma(c_{i_1}, c_{i_2}, c_{i_3})$. If $\s_i = \partial_\e \mathbf{P}_j \cap \partial_{\e'} \mathbf{P}_{j'}$, we discuss the gluing of $\mathbf S = \mathbf S_{j}$ and $\mathbf S' = \mathbf S_{j'}$ along the annuli $\mathbf{A} = \mathbf{A}_\e(\s_i) \subset \mathbf S$ and $\mathbf{A}' = \mathbf{A}_{\e'}(\s_i) \subset \mathbf S'$, that is we describe the complex projective structure on $\mathbf S \sqcup \mathbf S' / \sim$, where the equivalence relation $\sim$ is given by the attaching maps along the annuli $\mathbf{A}$ and $\mathbf{A}'$. 

Recall that there are strips 
$$\mathbf{H} = \mathbf{H}_\e = \boldsymbol\Omega_\e^{-1}(\mathbf{H}_\infty) \;\;\text{and}\;\;\mathbf{H}' = \mathbf{H}_{\e'} = \boldsymbol\Omega_{\e'}^{-1}(\mathbf{H}_\infty)$$ 
in $\HH \subset \hat\C$ such that $$\boldsymbol\pi(\mathbf{H}) = \mathbf{A} \subset \mathbf{S}\;\;\;\text{and}\;\;\;\boldsymbol\pi(\mathbf{H}') = \mathbf{A}' \subset \mathbf{S}',$$ where $$\boldsymbol\pi \co \HH \to \mathbb P_j = \mathbb P(c_{i_1}, c_{i_2}, c_{i_3}) = \HH/\Gamma(c_{i_1}, c_{i_2}, c_{i_3})$$ and $$\boldsymbol\pi' \co \HH \to \mathbb P'_j = \mathbb P'(c_{i_1}, c_{i_4}, c_{i_5}) = \HH/\Gamma(c_{i_1}, c_{i_4}, c_{i_5}).$$ So $\boldsymbol\pi \sqcup \boldsymbol\pi' \co \mathbf{H} \sqcup \mathbf{H}'  \to \mathbf{A} \sqcup \mathbf{A}'$. Let $$\boldsymbol\pi_{\mathbf{H}} \co \mathbf{H} \sqcup \mathbf{H}' \to \mathbf{H} \sqcup \mathbf{H}'/\sim\;\;\;\text{and}\;\;\;\boldsymbol\pi_{\mathbf{A}} \co \mathbf{A} \sqcup \mathbf{A}' \to \mathbf{A} \sqcup \mathbf{A}'/\sim.$$ With abuse of notation let's denote $$\boldsymbol\pi \sqcup \boldsymbol\pi' \co (\mathbf{H} \sqcup \mathbf{H}' /\sim)  \to (\mathbf{A} \sqcup \mathbf{A}'/\sim).$$ We can see that $\boldsymbol\pi_{\mathbf{A}} \circ (\boldsymbol\pi \sqcup \boldsymbol\pi') = (\boldsymbol\pi \sqcup \boldsymbol\pi') \circ \boldsymbol\pi_{\mathbf{H}}.$

Note that $V \subset \mathbf{H} \sqcup \mathbf{H}'/\sim$ is open if and only if $\boldsymbol\pi_{\mathbf{H}}^{-1}(V) \subset \mathbf{H} \sqcup \mathbf{H}'$ is open. Note also that $\boldsymbol\pi_{\mathbf{H}}^{-1}(V) = V_1 \sqcup V_2$, where $V_1 = \boldsymbol\pi_{\mathbf{H}}^{-1}(V) \cap \mathbf{H}$ and $V_2 = \boldsymbol\pi_{\mathbf{H}}^{-1}(V) \cap \mathbf{H}'$. Using the first fact above, we can see that there are natural complex projective structures on $\mathbf{H}$ and $\mathbf{H}'$, respectively, where the charts are the inclusion maps $i$. We define a complex projective structure on $\mathbf{H} \sqcup \mathbf{H}'/\sim$ as follows. Let $\cal V = \{V_i\}$ be a covering of $\mathbf{H} \sqcup \mathbf{H}'/\sim$, we define two sets of charts on $\cal V$: the charts $\psi^1 = \psi^1_V  \co V \to \mathbf{H} \subset  \hat\C$ are defined by $V \stackbin[]{\boldsymbol\pi_{\mathbf{H}}^{-1}}{\longmapsto}  V_1 \stackbin[]{i}{\mapsto} \mathbf{H}$ and the charts $\psi^2 = \psi^2_V \co V \to \mathbf{H}' \subset \hat\C$ defined by $V \stackbin[]{\boldsymbol\pi_{\mathbf{H}}^{-1}}{\longmapsto}  V_2 \stackbin[]{i}{\mapsto} \mathbf{H}',$ for every  $V \in \cal V$. The transition maps are given by $$\psi^1 \circ (\psi^2)^{-1} = \boldsymbol\Omega_{\epsilon}^{-1}\mathbf{J}^{-1}\mathbf{T}_{\boldsymbol\mu_i}^{-1}\boldsymbol\Omega_{\epsilon'} \in \PSL(2, \C)$$ for every $V \in \cal V$. This defines a complex projective structure on $\mathbf{H} \sqcup \mathbf{H}'/\sim$ which descends to a complex projective structure on $\mathbf{A} \sqcup \mathbf{A}'/\sim$ (see the second fact above). This defines a complex projective structure on $\mathbf{S} \sqcup \mathbf{S}' / \sim$.

Carrying out the above construction for each pants curve $\s_i \in \PC$ and considering the unique maximal atlas in the equivalence class of this atlas, we can define a complex projective structure on the quotient $(\mathbf{S}_1 \sqcup \ldots \sqcup \mathbf{S}_k)/\sim$, which we denote by $\Sigma(\uc, \umb)$. As described in Section \ref{sub:complex}, we can consider the developing map $${\rm Dev}_{\uc, \umb}\co\widetilde{\Sigma} \to \hat{\C}$$ and the holonomy representation $$\rho_{\uc, \umb}\co \pi_1(\Sigma) \to \PSL(2, \C)$$ associated to this complex projective structure. Both these maps are well defined up to the action of $\PSL(2,\C)$ and ${\rm Dev}_{\uc, \umb}$ is equivariant with respect to $\rho_{\uc, \umb}$.

As a consequence of the construction, we note the following fact.

\begin{Lemma}\label{hyper}
  If $\g \in \pi_1(\Sigma)$ is a loop homotopic to a pants curve $\s_i \in \PC$, then $\rho_{\uc, \umb}(\g)$ is hyperbolic with translation length $2c_i$.
\end{Lemma}

\subsection{The marking} \label{sec:mark}

To complete the description of the marked complex projective structure, we have to specify a marking on $\Sigma(\uc, \umb)$, that is, a homeomorphism $f_{\uc, \umb}\co \Sigma \to \Sigma (\uc, \umb)$ from a fixed topological surface $\Sigma$ to the surface $\Sigma(\uc, \umb)$. Endow $\Sigma$ with a marking decomposition, that is, a pants decomposition $\PC = \{\s_1, \ldots, \s_\xi\}$ and a set of dual curves $\{D_1, \ldots, D_\xi\}$ as described in Section \ref{sec:marking}. 

We first describe this marking for the particular case $\Sigma(\uc, \umb^0)$, where $\umb^0 = (\boldsymbol\mu_1^0, \ldots, \boldsymbol\mu_\xi^0)$ is defined by $\Re\boldsymbol\mu_i^0 = -c_i$, for all $i = 1, \ldots, \xi$, and then we see how to deal with all the other cases. The imaginary part of $\boldsymbol\mu_i^0$ is not important. (For definiteness, you can fix it to be $\Im\boldsymbol\mu_i^0 = 0$.) In particular, we describe a marking decomposition for $\Sigma(\uc, \umb^0)$ in such a way that we can define the homeomorphism $f_{\uc, \umb^0}\co \Sigma \to \Sigma(\uc, \umb^0)$ by asking the pants curves $\s_i$ and the dual curves $D_i$ to be mapped to the `corresponding' curves in $\Sigma(\uc, \umb^0)$. Let ${\boldsymbol\lambda}_{\e}(c_{i_1}, c_{i_2}, c_{i_3}) \subset {\boldsymbol\Delta}_0(c_{i_1}, c_{i_2}, c_{i_3})$ be the unique common perpendicular of the geodesics $\Ax\left(A_{\e +1}(c_{i_1}, c_{i_2}, c_{i_3})\right)$ and $\Ax\left(A_{\e +2}(c_{i_1}, c_{i_2}, c_{i_3})\right)$, where $\e$ is in the cyclically ordered set $\{0,1,\infty\}$, see Figure~\ref{fig9}. The lines ${\boldsymbol\lambda}_{\e}(c_{i_1}, c_{i_2}, c_{i_3})$ project to the seams of $\mathbf{P} = \mathbf{P}(c_{i_1}, c_{i_2}, c_{i_3})$. We call the geodesics ${\boldsymbol\lambda}_0$ (from $\Ax(A_{1})$ to $\Ax(A_{\infty})$) and ${\boldsymbol\lambda}_1$  (from $\Ax(A_{\infty})$ to $\Ax(A_{0})$) the \emph{incoming} and the \emph{outgoing} geodesic, respectively, and refer to their images under the maps ${\boldsymbol\Omega}_{\e}$ in a similar way. For $\boldsymbol\mu \in \C/2i\pi$, let $$X_{\infty} (\boldsymbol\mu)\;\;\;\text{and}\;\;\;Y_{\infty}(\boldsymbol\mu)$$ be the points at which the incoming line ${\boldsymbol\lambda}_0$ and the outgoing line ${\boldsymbol\lambda}_1$ meet the hypercycle $\mathbf{h}_{\infty, \HH}(c_{i_1}, c_{i_2}, c_{i_3})$ in $\HH$. Also define $$X_{\e}(\boldsymbol\mu)=\boldsymbol\Omega_{\e}^{-1}\left(X_{\infty}(\boldsymbol\mu)\right)\;\;\;\text{and}\;\;\;Y_{\e}(\boldsymbol\mu)=\boldsymbol\Omega_{\e}^{-1}\left(Y_{\infty}(\boldsymbol\mu)\right).$$ Note that $X_0(\boldsymbol\mu), Y_0(\boldsymbol\mu) \in \mathbf{h}_0(c_{i_3}, c_{i_1}, c_{i_2})$ and $X_1(\boldsymbol\mu), Y_1(\boldsymbol\mu) \in \mathbf{h}_1(c_{i_2}, c_{i_3}, c_{i_1})$. 

Now pick a pants curve $\s$ and let  $P,P' \in \P$ be its adjacent pants in $\Sigma$ to be glued across boundaries $\dd_{\e}P$ and $\dd_{\e'}P'$.  Let 
$$X_{\infty}(P, \boldsymbol\mu)\;\;\;\text{and}\;\;\;X_{\infty}(P', \boldsymbol\mu)$$ be the points in $\mathbf{P} = \mathbf{P}(c_{i_1}, c_{i_2}, c_{i_3})$ corresponding to $X_{\infty}(\boldsymbol\mu)$ under the identifications $(\boldsymbol\pi|_{{\boldsymbol\Delta}})^{-1}$ and $(\boldsymbol\pi'|_{{\boldsymbol\Delta}'})^{-1}$ of ${\boldsymbol\Delta}$ and  ${\boldsymbol\Delta}'$ with $\mathbf{P}$. Define similarly the points:
\begin{itemize}
  \item[--] $Y_{\infty}(P, \boldsymbol\mu), Y_{\infty}(P', \boldsymbol\mu) \in \mathbf{P}(c_{i_1}, c_{i_2}, c_{i_3}),$
  \item[--] $X_{0}(P, \boldsymbol\mu), X_{0}(P', \boldsymbol\mu), Y_{0}(P, \boldsymbol\mu), Y_{0}(P', \boldsymbol\mu) \in \mathbf{P}(c_{i_3}, c_{i_1}, c_{i_2})$, and
  \item[--] $X_{1}(P, \boldsymbol\mu), X_{1}(P', \boldsymbol\mu), Y_{1}(P, \boldsymbol\mu), Y_{1}(P', \boldsymbol\mu) \in \mathbf{P}(c_{i_2}, c_{i_3}, c_{i_1}).$
\end{itemize}

The base structure $\Sigma(\uc, \underline{\boldsymbol\mu}^0)$ is the one in which the identification~\eqref{eqn:gluing_bold} matches the point $X_{\e}(P,\boldsymbol\mu)$ on the incoming line across $\dd_{\e}\mathbf{P}$ to the point $Y_{\e'}(P')$ on the outgoing line to $\dd_{\e'} \mathbf{P}'$. Referring to the gluing equation~\eqref{eqn:gluing_bold} and to Figure \ref{fig10}, we see that this condition is fulfilled precisely when $\Re\boldsymbol\mu_i = -c_i$. We define the structure on $\Sigma$ by specifying $\Re\boldsymbol\mu_i = -c_i$ for $i = 1,\ldots,\xi$. This explains our choice for $\umb^0$. The imaginary part of $\boldsymbol\mu_i$ is unimportant for the above condition to be true. 

Now observe that the map $\mathbf{J}$ of $\HH$ induces an orientation-reversing isometry of $\mathbf{P}(c_{i_1}, c_{i_2}, c_{i_3})$ which fixes its seams; with the gluing matching the seams, as above, this extends, in an obvious way, to an orientation-reversing involution of $\Sigma$. Following Section~\ref{sec:marking}, this specification is equivalent to a specification of a marking decomposition on $\Sigma$. This gives us a way to define the marking $f_{\uc, \umb^0}\co \Sigma \to \Sigma (\uc, \umb^0)$.

Finally, we define a marking on the surface $\Sigma(\uc, \umb)$. After applying a suitable stretching to each pants in order to adjust the lengths of the boundary curves, we can map $\Sigma(\uc, \umb^0) \to \Sigma(\uc,\umb)$ using a map which is the Fenchel--Nielsen twist ${\Tw}_{\s_i} (\Re \mu_i + c_i)$ on an annulus around $\s_i \in \PC, i=1,\ldots, \xi$ and the identity elsewhere, see Section~\ref{sec:dehntwists} for the definition of ${\Tw}_{\s, t}$. This gives a well defined homotopy class of homeomorphisms $f_{\uc, \umb} \co \Sigma \to \Sigma(\uc, \umb)$\label{a37}. The stretch map used above depends on $\Im{\boldsymbol\mu}_i$.
 
As before, with this description, it is easy to see that $\Re {\boldsymbol\mu}_i$ corresponds to twisting about $\s_i$; in particular, ${\boldsymbol\mu}_i \mapsto {\boldsymbol\mu}_i + 2c_i$ is a full right Dehn twist about $\s_i$. The imaginary part $\Im {\boldsymbol\mu}_i$ corresponds to vertical translation and has the effect of scaling the lengths of the $\s_i$ in the complex projective structure.

\subsection{Relation with the c--plumbing construction}\label{plumb}

The gluing construction described by the author and Series \cite{mal_top} was essentially the same as Kra's plumbing construction \cite{kra_hor}. One advantage of our construction, in addition to providing simpler formulas, is that it respects the twisting around the punctures. In this section we want to compare the $\uc$--plumbing construction defined by Parker and Parkonnen \cite{par_coo} with the $\uc$--gluing construction described above. The advantages of the second description mentioned above remain.

Let's fix a marking decomposition $(\PC, \DD)$ of $\Sigma$ which consists of a pants decomposition $\PC = \{\s_1, \ldots, \s_\xi\}$ and a set of dual curves $\DD = \{D_1, \ldots, D_\xi\}$. If $\s_{i_k}$ is a pants curve, let $c_{i_k}$ be the positive real number given by $\uc \in \R_+^\xi$; if $\s_{i_k}$ is an element of $\partial \Sigma$, let $c_{i_k} = 0$. Every pair of pants $P \in \P$ has three boundary components $\s_{i_1}, \s_{i_2}$ and $\s_{i_3}$ which could be pants curves in $\PC$ or punctures in $\partial \Sigma$. We can identify $P$ with the `standard' pair of pants $\mathbf{P} = \mathbf{P}(c_{i_1},c_{i_2},c_{i_3}) = \HH/\boldsymbol\Gamma(c_{i_1},c_{i_2},c_{i_3})$ described in Section~\ref{sec:std}. 

Suppose that the pairs of pants $P, P' \in \P$ are adjacent along the pants curve $\s = \s_i$ of length $l = 2c_i = c > 0$. Let $${\boldsymbol\Phi}\co \mathrm{Int}(P) \to \mathbf{P} = \mathbf{P}(c,c_{i_2},c_{i_3})\;\;\text{and}\;\;{\boldsymbol\Phi}'\co \mathrm{Int}(P') \to \mathbf{P}' = \mathbf{P}(c,c_{i_4},c_{i_5})$$ be the identification maps described in Section \ref{sec:cglu}. Suppose $\partial_\infty\mathbf{P}$ and $\partial_\infty\mathbf{P}'$ are the boundary components corresponding to $\s$. Let $D$ be the curve dual to $\s$ and let $$\A_{c} = \{z \in \C: e^{-\frac{\pi^{2}}{c}} < |z| < 1 \}$$ be an annulus of modulus $$\mathrm{mod}(\A_c) = \frac{1}{2\pi}\log{e^{\frac{\pi^{2}}{c}}} = \frac{\pi}{2c}$$ with its hyperbolic metric of constant curvature $-1$. The curve $\{z \in \C: |z| = e^{-\frac{\pi^{2}}{2c}}\}$ is the unique geodesic in $\A_{c}$ for this metric and has length $2c$. 

Let $U \subset \mathbf{P}$ and $U'\subset \mathbf{P}'$ be annular neighbourhoods of the ends of $\mathbf{P}$ and $\mathbf{P}'$, respectively, corresponding to the boundary $\partial_\infty\mathbf{P}$ and $\dd_\infty\mathbf{P}'$. We define local coordinates 
$$\hat z\co U \to \A_{c} \hspace{5mm} {\rm and} \hspace{5mm} \hat z'\co U' \to \A_{c}$$ 
by requiring that the maps $\hat z$ and $\hat z'$ are isometries and that the segments $D \cap U$ and $D \cap U'$ are mapped into $\A_{c} \cap \R$. (If $D \cap U$ and $D \cap U'$ are connected, we require them to be mapped into $\A_{c} \cap \R_+$.) These conditions uniquely define the maps $\hat z$ and $\hat z'$.

 Assume there are annuli $\A \subset U$ and $\A' \subset U'$ and assume there is a holomorphic homeomorphism $f\co \A \to \A'$ which maps the outer boundary of $\A$ to the inner boundary of $\A'$ and so that, for some complex parameter $\mathbf{t} \in \C$, we have $$\hat z(x)\hat z'\left(f(x)\right) = \mathrm{exp}(i\pi\mathbf{t}) \hspace{5mm}\forall x \in \A.$$ (Usually, this relationship is written $\hat z \hat z' = t$, but we have changed the notation since it will make the future results easier to state.) 
 
The outer boundaries of $\A$ and $\A'$ bound annuli on $\mathbf{P}$ and $\mathbf{P}'$. Remove these (closed) annuli to form two new Riemann surfaces $P_{trunc}$ and $P'_{trunc}$, and define an equivalence relation on it by setting, for all $x \in P_{trunc}, y \in P_{trunc}'$:
$$x \sim y \iff \hat z(x)\hat z'(y) = \mathrm{exp}(i\pi\mathbf{t}),$$ 
where $\mathbf{t} \in \C$. Let $X_{\mathbf{t}} = P_{trunc}\sqcup P'_{trunc} /\sim$. We say that $X_{\mathbf{t}}$ was obtained from $P$ and $P'$ by the $c$--plumbing construction with plumbing parameter $\mathbf{t} = \mathbf{t}_i$. Repeating this construction for all the pants curves in $\PC$, we can describe the $\uc$--{\it plumbing construction with plumbing parameters} $\utb = (\mathbf{t}_1, \ldots, \mathbf{t}_\xi) \in \left(\C\right)^\xi$.

\begin{Remark}
  Note that we can generalize this construction to the case that some parameters $c_i$ are zeroes. When $c = 0$, the annulus $\A_c$ becomes the punctured disk $\D^* = \D - \{0\}$ and the surface which is created has a puncture homotopic to $\s$. In that case, we define groups on the boundary of the quasi-Fuchsian space, and no longer in its interior. Kra's plumbing construction is obtained with a similar method, but using punctured disks $\D^*$, that is, we can consider it to be `associated' to the parameter $\uc = \underline 0 = (0, \ldots, 0)$. 
\end{Remark}

In Section \ref{sec:std}, supposing that we are gluing $\partial_\infty(P)$ to $\partial_\infty(P'),$ we described local coordinates $$z\co U \to \mathfrak{A} \hspace{5mm} {\rm and} \hspace{5mm}z'\co U' \to \mathfrak{A},$$ where 
$$\mathfrak{A} = \{z \in \HH : |z - C_1|> R_1, |z + C_1|>R_1\},$$ 
 $C_1 = \frac{\cosh c_{i_1}}{\cosh c_{i_1} -1}$ and $R_1 = \frac{1}{\cosh c_{i_1} -1}$. These two circles correspond to the isometric circles of the matrix $A_\infty = \left(
\begin{array}{cc}
  \cosh c  & \cosh c+1 \\
  \cosh c-1 & \cosh c \\
\end{array}
\right)$ and of its inverse $A_\infty^{-1}$. The region $\mathfrak{A}$ is a fundamental region for the action of the matrix $A_\infty$ on $\HH$. In addition, if $\boldsymbol\mu \in \C_{\pi}$ is the gluing parameter, we describe the gluing between $x \in \A$ and $y \in \A'$ as:
\begin{equation}\label{glue_mal}
  z'(y) = \mathbf{T}_{\boldsymbol\mu} \circ \mathbf{J}\; z(x).
\end{equation}

The parameters $z, z'$ and $\hat z, \hat z'$ are related by the following result.

\begin{Proposition}\label{glue_plum}
  The $\uc$--gluing construction with parameter $\umb \in (\C_{[0,\pi)})^\xi$ is a $\uc$--plumbing construction with parameter $\utb \in \HH^\xi$ and the parameters are related by:
  \begin{equation}
    \mathbf{t}_i = \frac{i\pi - {\boldsymbol\mu}_i}{c_i}, \;\;\forall i = \{1, \ldots, \xi\}.
  \end{equation}
\end{Proposition}

\begin{figure}
\centering
\includegraphics[height=13.5cm]{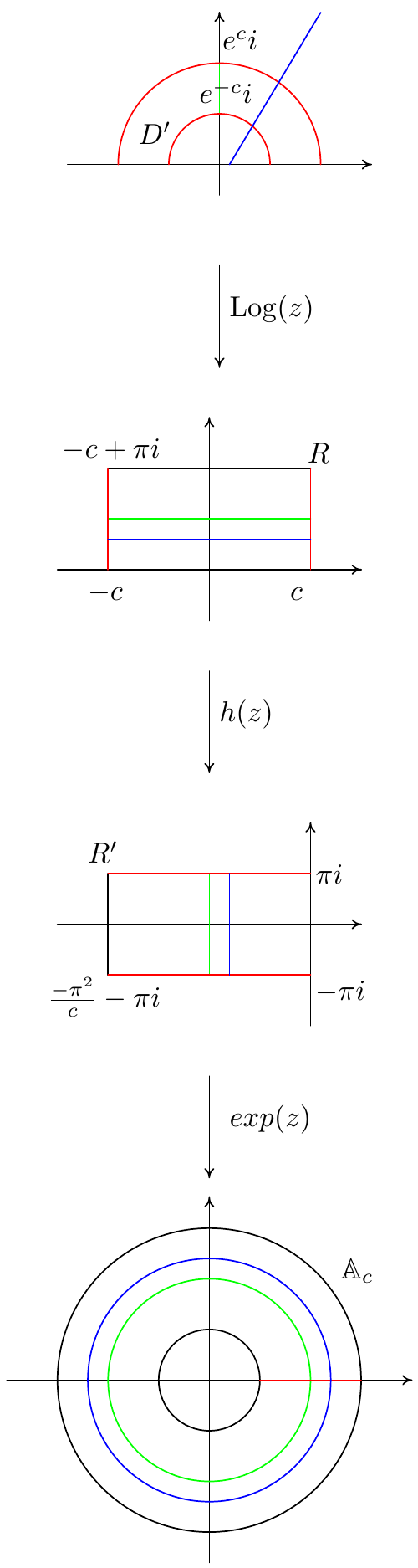}
\caption{The map $H \co \mathfrak{A}' \to \A_c$.}
\label{H}
\end{figure}

\begin{proof}
First, we want to find a function which maps $\mathfrak{A}$ into $\A_c$ and which sends $\Ax(A_\infty)$ to the geodesic $\{z \in \C: |z| = e^{-\frac{\pi^{2}}{2c}}\} \subset \A_c$.
This is done in two steps:
\begin{itemize}
  \item In the first step, use $M = \left(
  \begin{array}{cc}
  \sinh \frac{c}{2} & \cosh\frac{c}{2} \\
  -\sinh\frac{c}{2} & \cosh \frac{c}{2} \\
  \end{array}
  \right)$ to map $\Ax(A_\infty)$, which is the geodesic between $-\coth\frac{c}{2}$ and $\coth\frac{c}{2}$, to the imaginary axis. Note that the set $$M(\mathfrak{A}) = \mathfrak{A}' := \{z \in \HH : e^{-c} < |z| < e^{c} \}$$ is a fundamental domain for the action of the map $C= \left(
  \begin{array}{cc}
  e^{c} & 0 \\
  0 & e^{-c} \\
  \end{array}
  \right)$ on $\HH$.
  \item In the second step, use the map $H \co \mathfrak{A}' \to \A_c$ defined by $$H(z) = \exp(\frac{i\pi}{c} \mathrm{Log}z),$$ where:
  \begin{itemize}
    \item The principal value $\mathrm{Log}$ of the complex logarithm (that is, the branch with argument in $(-\pi, \pi]$) maps $\mathfrak{A}'$ into the rectangle $R = (-c,c) \times (0, \pi)$.
    \item The map $h$ defined by $h(z)=\frac{\pi i z}{c}$ maps $R$ into $R' = (\frac{-\pi^2}{c},0) \times (-\pi, \pi)$.
    \item The map $\exp$ maps $R'$ into $\A_c$.
  \end{itemize}
  See Figure \ref{H}.
\end{itemize}
Note also that $H(C(z)) = H(z)$, so this map extends to a well defined map $\mathfrak{A} \to \A_c$.

Now we can see the relationship between the coordinate $\hat z$ and the coordinate $z$:
\begin{equation}\label{map_H}
 \hat z =  \exp\left(\frac{i\pi}{c}\mathrm{Log}(M z)\right). 
\end{equation}

Second, we want to find the relationship between the $\uc$--gluing parameter $\boldsymbol\mu$ and the $\uc$--plumbing parameter $\mathbf{t}$. The relation $$\hat z \hat z' = \mathrm{exp}(i\pi\mathbf{t})$$ translates, using \eqref{map_H}, to:
\begin{align}\label{ct}
         \exp\left(\frac{i\pi}{c}\mathrm{Log}(M(z) M(z'))\right) &= \mathrm{exp}(i\pi\mathbf{t}) \notag\\
         \mathrm{Log}(M(z) M(z')) &= c\mathbf{t} \notag\\
         M(z) M(z') &= \exp(c\mathbf{t}).
\end{align}
Now, we can see that 
\begin{align*}
        M(z') &= M(\mathbf{T}_{\boldsymbol\mu} \circ \mathbf{J} (z)) \\
              &= \exp(-\boldsymbol\mu)\frac{\sinh \frac{c}{2} z - \cosh \frac{c}{2}}{\sinh \frac{c}{2} z + \cosh \frac{c}{2}},
\end{align*}
where the maps $\mathbf{T}_{\boldsymbol\mu}$ and $\mathbf{J}$ are defined in Equation \eqref{eqn:standardsymmetries1_bold}. On the other hand, $$M(z) = \frac{\sinh \frac{c}{2} z + \cosh \frac{c}{2}}{-\sinh \frac{c}{2} z + \cosh \frac{c}{2}}.$$
Hence  we have that $$M(z) M(z') = -\exp(-\boldsymbol\mu),$$ which corresponds, using Equation \eqref{ct}, to $$\exp(i \pi-\boldsymbol\mu) = \exp(c\mathbf{t}).$$ 
This gives the formula we want to prove.
\end{proof}

\begin{Remark}[Relationship with Fenchel-Nielsen construction]\label{FN-cons}
The Fenchel-Nielsen construction describes how, given a point $(\frac{\l_1}{2},\ldots, \frac{\l_\xi}{2}, \t_1,\ldots \t_\xi) \in \FN_{\C}(\QF)$ one can build a Kleinian group $G \cong \pi_1(\Sigma)$ such that $\FN_{\C}(G) = (\frac{\l_1}{2},\ldots, \frac{\l_\xi}{2}, \t_1,\ldots \t_\xi)$. In \cite{par_coo} the authors prove that in the case of the once-punctured torus $\Sigma_{1, 1}$ the Fenchel-Nielsen construction corresponds to a $c$--plumbing construction (and so to a $c$--gluing construction). That result can be generalized to our setting as well. In Appendix \ref{app:example} we will discuss the precise calculations in the case of a once-punctured torus $\Sigma_{1, 1}$ (HNN-extension) and of a four-punctured sphere $\Sigma_{0,4}$ (amalgamated free product AFP construction).
\end{Remark}

\subsection{Relation to $\QF(\Sigma)$}\label{sub:QF}

In this section, we want to prove the following:

\begin{Theorem}\label{QF}
  Suppose $\uc \in \R_+^\xi$ and $\umb \in \C_{[0,\pi)}^\xi$ are so that the developing map ${\rm Dev}_{\uc, \umb}$ associated to the complex projective structure $\Sigma(\uc, \umb)$ is an embedding, then the holonomy representation $\rho_{\uc, \umb}$ is in $\QF(\Sigma)$ and the holonomy images $\rho_{\uc, \umb}(\s_i)$ of the pants curves are hyperbolic elements with translation length $2c_i$ for $i = 1, \ldots, \xi$. 
\end{Theorem}

\begin{proof}
  Since ${\rm Dev}_{\uc, \umb}$ is an embedding, then $\Gamma_{\uc, \umb} = \rho_{\uc, \umb}(\pi_1(\Sigma))$ is Kleinian, and by construction (see Lemma~\ref{hyper}) the holonomy images of each curve $\s_1, \ldots, \s_\xi$ is hyperbolic with translation length $2c_i$. In addition, the hypothesis that ${\rm Dev}_{\uc, \umb}$ is an embedding enforces the hypercycles used in the $\uc$--gluing construction to be disjoint. (Note that one can also find precise conditions on the parameters $\uc, \umb$ that imply the same statement, as done in Theorem 4.1 of \cite{par_coo}, but in our case it will be much more complicated and we do not think will be particularly useful.) In this situation, we can use Maskit's Combination Theorems \cite{mas_kle} to show that $\Gamma_{\uc, \umb}$ has a fundamental domain with two components and each component glues up to give a surface homeomorphic to $\Sigma$, so $\Gamma_{\uc, \umb}$ is quasi-Fuchsian.
\end{proof}

\begin{Remark}
  We can prove even more, as done in Theorem 4.2 in \cite{par_coo}. In the hypothesis of Theorem \ref{QF} one can see that the support of the pleating locus on the `bottom' surface is $\PC$ and $\umb$ corresponds to the complex shear along the pants curves. In this result, it is important that we are bending in the same direction along all the curves, so the result we obtain is convex. Of course all the discussion could be generalized to the case $\umb \in (\C_{(-\pi, 0]})^\xi$, where $\C_{(-\pi, 0]} = \{ z \in \C \mid \mathrm{Im}(z) \in (-\pi, 0]\}$.
\end{Remark}

\section{Limits of c--plumbing structures} \label{sec:lim}

In this section we want to prove Theorem \ref{thmE}, that is we want to show that, when $\uc$ tends to $\underline{0} = (0, \cdots, 0)$, the $\uc-\text{plumbing}$ construction of Section \ref{sub:c_glu} limits to the gluing construction described in \cite{mal_top}. A corollary to Theorem \ref{thmE} which follows from the topology on the set of complex projective structure is the following result:
\begin{Corollary}
  Let $\uc = (c_{1},\ldots, c_{\xi}) \in \R_{+}^{\xi}$ and $\umb = ({\boldsymbol\mu}_1, \ldots, {\boldsymbol\mu}_\xi)\in (\C_{[0, \pi)})^\xi$. If $\uc \to \underline{0}$ keeping $\um  = (\mu_1, \ldots, \mu_\xi)$ fixed, where $\mu_i = \frac{i\pi-{\boldsymbol\mu}_i}{c_i}$, then the sequence of holonomy representations $(\rho_{\uc, \umb})$ defined by the $\uc$--gluing construction converges to the holonomy representation $\rho_\um$ defined by the gluing constriction of \cite{mal_top}. 
\end{Corollary}

We start by recalling this last construction and then prove the theorem in the following section. 

\subsection{Maloni and Series' gluing construction}\label{mal-ser}

The gluing construction we described in \cite{mal_top} is based on Kra's plumbing construction~\cite{kra_hor}. It describes a projective structure on $\Sigma$ obtained by gluing triply punctured spheres across their punctures. More precisely, given a pants decomposition $\P = \{P_1, \ldots, P_k\}$ on $\Sigma$, we fix an identification $\Phi_i \co \mathrm{Int}(P_i) \to \mathbb P$ of the interior of each pair of pants $P_i$ to the standard triply punctured sphere $\mathbb P = \HH /\Gamma$, where $$\Gamma = 
\Bigl <\begin{pmatrix}
	1  &  2  \\
	0 &  1 \\ 
	\end{pmatrix},
	 	\begin{pmatrix}
	1  &  0  \\
	2 &  1 \\ 
	\end{pmatrix} \Bigr >.$$ This identification induces a labelling of the three boundary components of $P_{i}$ as $0, 1, \infty$ in some order. We endow $\mathbb P$ with the projective structure coming from the unique hyperbolic metric on a triply punctured sphere. 
	
	\begin{figure}[hbt] 
  \centering 
  \includegraphics[height=15cm]{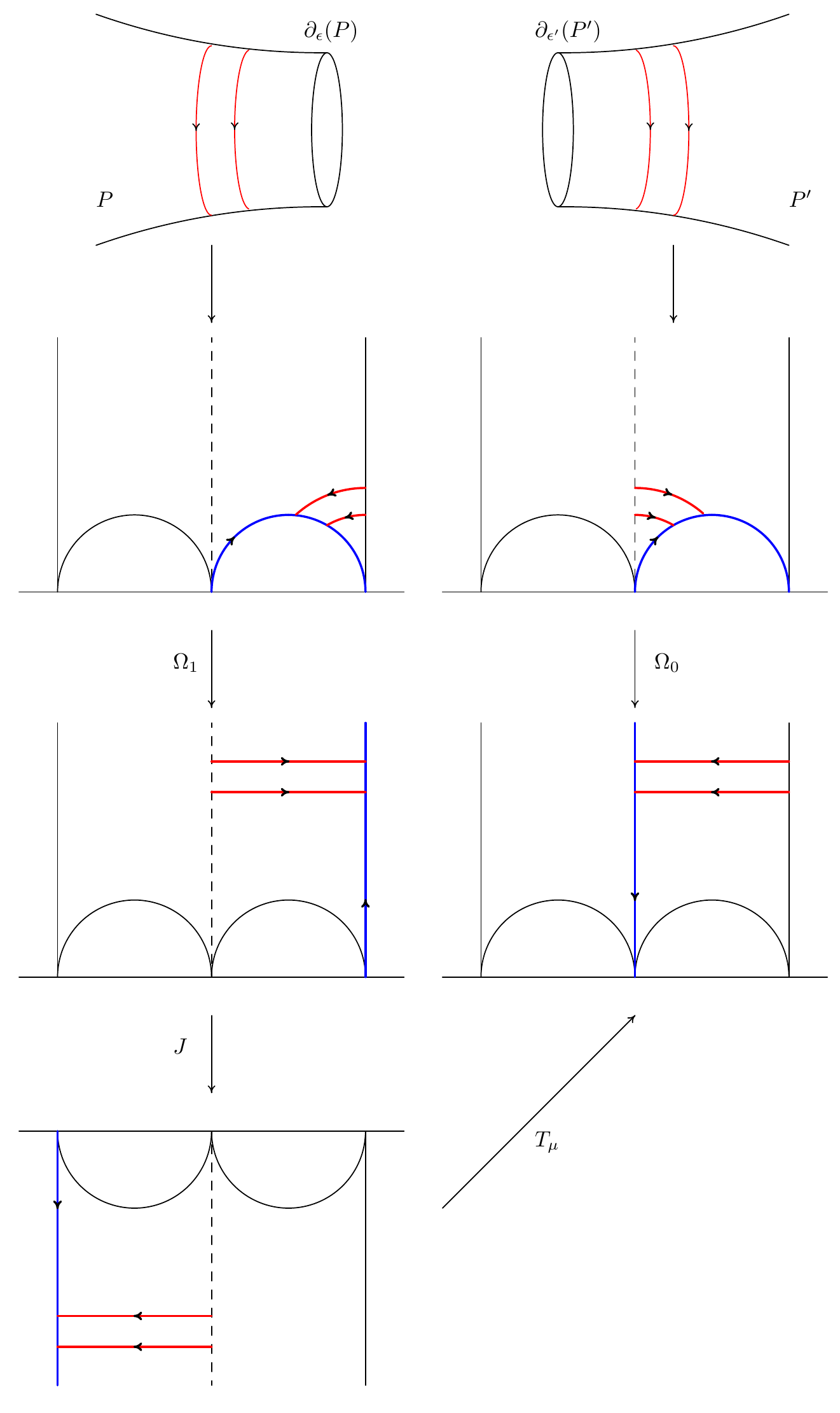} 
  \caption{The gluing construction from \cite{mal_top} in the case $\e = 1$ and $\e' = 0$.}
  \label{figure4_10}
  \end{figure}
  
  Suppose that the pants $P, P' \in \P$ are adjacent along the pants curve $\s$ and suppose that $\s$ corresponds to $\dd_{\e}P$ and $\dd_{\e'}P'$. The gluing is obtained by deleting open punctured disk neighbourhoods of the two punctures in question and gluing horocyclic annular collars round the resulting boundary curves. The gluing across $\s = \s_j$ is described by a parameter $\mu = \mu_j \in \HH$. 
  
  As before, we first describe the gluing when $\e = \e' = \infty$. Consider two copies $\mathbb P, \mathbb P'$ of $\mathbb P$. They are both identified with $\HH/\Gamma$. We refer to the copy of $\HH$ associated to $\mathbb P'$ as $\HH'$ and denote the natural parameters in $\HH$ and $\HH'$ by $z$ and $z'$ respectively. Let $\pi$ and $\pi'$ be the projections $\pi\co\HH \to \mathbb P$  and  $\pi'\co \HH' \to \mathbb P'$ respectively. Let $h= h(\mu) $ be the simple loop on $\mathbb P$ which lifts  to the horocycle $\{z \in \HH | \Im z= \frac{\Im \mu}{2}\}$ on $\HH$. Given $\eta>0$, let $$H = H(\mu,\eta) = \{z \in \HH \mid \frac{\Im \mu-\eta}{2} \leqslant \Im z \leqslant  \frac{\Im \mu+\eta}{2}\} \subset \HH$$ be a horizontal strip which projects to an annular neighbourhood $A$ of $h \subset \mathbb P$. Let $S \subset \mathbb P$ be the surface obtained from $\mathbb P$ by removing the open horocyclic neighbourhood $\{z \in \HH \mid \Im z >  \frac{\Im \mu+\eta}{2}\}$ of $\infty$. Define $h', S'$ and $A'$ in a similar way.  We will glue $S$ to $S'$ by matching $A $ to $A'$  so that $h$ is identified to $h'$ with orientation reversed, see Figures~\ref{figure4_10}. Let
  \begin{equation}
  \label{eqn:standardsymmetries1}
  J = \begin{pmatrix}
  	-i  &  0  \\
  	0 &  i   
  	\end{pmatrix},  \ \ 
  T_{\mu}=  \begin{pmatrix}
  1 & \mu  \\
  	0 &  1 \end{pmatrix}.
  \end{equation}
   We glue the strips $H$ to $H'$ by  identifying $z \in H$ to $z' = T_{\mu} J(z) \in H'$. This identification descends to a well defined identification of the annuli $A$ with $A'$, where the `outer' boundary of $A$ is identified to the `inner' boundary of $A'$ and $h$ is glued to $h'$ reversing the orientation. 
   
   In the general case, first `move' $\e$ and $\e'$ to $\infty$ using the maps \begin{equation}
   \label{eqn:standardsymmetries}
   \Omega_{0} = \begin{pmatrix}
   	1  &  -1  \\
   	1 &  0   
   	\end{pmatrix},  \ \ 
    \Omega_{1} =  \begin{pmatrix}
   0 &  -1  \\
   	1 &  -1 \end{pmatrix},  \ \ 
   	\Omega_{\infty} = \Id = \ \begin{pmatrix}
   	1  &  0  \\
   	0 &  1 	\end{pmatrix}.
   \end{equation} and then proceed as before. 
   
   Finally, we perform this construction for each pants curve $\s_i \in \PC$. In order to do this, we need to ensure that, for every pair of pants $\mathbb P_i$, the annuli corresponding to the three different punctures are disjoint. We can choose $\mu_i$ and $\eta$ so that this is true. This allows us to describe a complex projective structure on the quotient $$S(\underline{\mu}) = S_1 \sqcup \ldots \sqcup S_k/\sim$$ (homeomorphic to $\Sigma$), where the equivalence relation $\sim$ is given by the attaching maps along the annuli $A_i$.

\subsection{Proof of Theorem \ref{thmE}} \label{subsec:lim}

We now want to prove Theorem \ref{thmE}. We will show that the developing maps $\mathrm{Dev}_{\uc, \umb}$ (associated with the complex projective structure $\Sigma(\uc, \umb)$) converges locally uniformly to the developing maps $\mathrm{Dev}_{\um}$ (associated with $\Sigma(\um)$). We identify the universal cover of $\Sigma$ with the set of homotopy classes of paths on $\mathbf{S}(\uc, \umb)$ and $S(\um)$ respectively and then calculate the images of these path by following the recipe described by the gluing construction of \cite{mal_top} (see also the next section) and the $\uc$--gluing construction described in the previous section. To do this we have to describe the developing image of any path $\g$ on $S(\underline{\mu})$ and $\mathbf{S}(\uc, \umb)$.

\begin{proof}[Proof of Theorem \ref{thmE}]
First we show that, when $c_{i_1},c_{i_2}$ and $c_{i_3}$ tend to $0$, the group $\boldsymbol\Gamma(c_{i_1}, c_{i_2}, c_{i_3})$ limits to the group $\Gamma = \langle\left(
\begin{array}{cc}
  	1  &  -2  \\
  	0 &  1 \\ 
  	\end{array}
    \right), \left(
    \begin{array}{cc}
	1  &  0  \\
	2 &  1 \\ 
	\end{array}
  \right) \rangle$ in the sense that each generator does:  When $c_{i_1}$ tends to $0$, then $\cosh c_{i_1}$ tends to $1$. This shows that the matrix $$A_\infty = A_\infty(c_{i_1}, c_{i_2}, c_{i_3}) \to \left(
  \begin{array}{cc}
  1  & 2 \\
  0 & 1 \\
  \end{array}
  \right) = \left(
  \begin{array}{cc}
    	1  &  -2  \\
    	0 &  1 \\ 
    	\end{array}
      \right)^{-1}.$$ 
For describing the other two limits we need to use the Taylor expansion of the hyperbolic functions. In particular, we use: $$\sinh x = x + O(x^3), \;\; \cosh x = 1 + O(x^2),\; \;\tanh x = x + O(x^3),\;\;\coth x = \frac{1}{x} + O(x).$$ We  use the symbol $\approx$ when we mean `up to terms of lower degree'. Using the Taylor approximations described above and Equation \eqref{nu}, you can see that $\nu_1 \approx \frac{c_{i_1} c_{i_2}}{2}$, while $\nu_2 \approx \frac{c_{i_1} c_{i_3}}{2}$. Hence, as $c_{i_1},c_{i_2},c_{i_3} \to 0$, we have that: $$-\coth(\frac{c_{i_1}}{2})\tanh(\frac{\nu_1}{2}) \sinh c_{i_2} \to 0, \;\; -\tanh(\frac{c_{i_1}}{2})\coth(\frac{\nu_1}{2}) \sinh c_{i_2} \to -2.$$ So $$A_0 = A_0(c_{i_1}, c_{i_2}, c_{i_3}) \to \left(
  \begin{array}{cc}
  1  & 0 \\
  -2 & 1 \\
  \end{array}
  \right)= \left(
  \begin{array}{cc}
1  &  0  \\
2 &  1 \\ 
\end{array}
\right)^{-1},$$
  as we wanted to prove.
  
  This implies that $\Ax\left(A_\infty\right)$ limits to the fixed point $\Fix(\left(
  \begin{array}{cc}
    	1  &  -2  \\
    	0 &  1 \\ 
    	\end{array}
      \right)) = \infty$ and, similarly, $\Ax\left(A_0\right)$ limits to the fixed point $\Fix(\left(
      \begin{array}{cc}
    1  &  0  \\
    2 &  1 \\ 
    \end{array}
    \right))$ and hence $\Ax\left(A_1\right) \to 1 = \Fix(\left(
    \begin{array}{cc}
  3  &  -2  \\
  2 &  -1 \\ 
  \end{array}
  \right))$. In addition, we can see that the hypercycle $\mathbf{h}_{\infty, \HH}(c_{i_1}, c_{i_2}, c_{i_3})$ defined in Section \ref{sec:cglu} limits to the horocycle $h_{\infty, \HH} = \zeta(h_\infty) = \{z \in \HH | \Im z = \frac{\mu_{i_1}}{2}\}$ defined in \cite{mal_top} by analyzing the limit of the point of intersection between the imaginary axis and $\mathbf{h}_{\infty, \HH}(c_{i_1}, c_{i_2}, c_{i_3})$, that is, of the point $i\coth(\frac{c_{i_1}}{2})\left(-\tan(\th)+\frac{1}{\cos(\th)}\right)$, where $\th = \frac{\Im{\boldsymbol\mu}_{i_1}}{2}$. We can do that by using the Taylor expansion of $\coth$ and $\tan$, and the identity $-\tan(\th)+ \mathrm{sec}(\th) = \tan(-\frac{\th}{2}+\frac{\pi}{4})$. We have 
    $$i\coth(\frac{c_{i_1}}{2})\left(-\tan(\th)+\frac{1}{\cos(\th)}\right) \approx i\frac{2}{c_{i_1}}\frac{1}{2}\left(\frac{-\Im{\boldsymbol\mu}_{i_1} +\pi}{2}\right) \to i \frac{\Im \mu_{i_1}}{2},$$
    as we wanted to prove.
 
 Finally, we need to prove that the gluing maps agree, that is, we need to prove that when $\uc \to \underline{0}$ keeping $\um  = (\mu_1, \ldots, \mu_\xi)$ fixed,
 $$\boldsymbol\Omega_{\epsilon}^{-1}\mathbf{J}^{-1}\mathbf{T}_{\boldsymbol\mu}^{-1}\boldsymbol\Omega_{\epsilon'} \to \Omega_{\epsilon}^{-1}J^{-1}T_\mu^{-1}\Omega_{\epsilon'}.$$
 First consider the case $\e = \e' = \infty$, where $\boldsymbol\Omega_{\infty} = \Omega_{\infty} = \Id$. We have:  \begin{align*}
  \rho_{\uc, \umb}(\vartheta) = \mathbf{J}^{-1}\mathbf{T}_{\boldsymbol\mu}^{-1} &= \begin{pmatrix}
    	0  &  \coth(\frac{c_i}{2})  \\
    	-\tanh(\frac{c_i}{2}) &  0   
    	\end{pmatrix}  \begin{pmatrix}
    \cosh \frac{\boldsymbol\mu}{2} & \sinh\frac{\boldsymbol\mu}{2}\coth \frac{c_i}{2} \\
    \sinh\frac{\boldsymbol\mu}{2}\tanh \frac{c_i}{2} & \cosh\frac{\boldsymbol\mu}{2} \\
    \end{pmatrix}\\[0.5ex] 
  & =   \begin{pmatrix}
     \sinh\frac{\boldsymbol\mu}{2}  &  \cosh\frac{\boldsymbol\mu}{2}\coth \frac{c_i}{2} \\
       -\cosh \frac{\boldsymbol\mu}{2}\tanh \frac{c_i}{2} &  -\sinh\frac{\boldsymbol\mu}{2}  
      \end{pmatrix}\\[0.5ex] 
      & =   i\begin{pmatrix}
            \cosh \frac{\mu c_i}{2} & -\sinh\frac{\mu c_i}{2}\coth \frac{c_i}{2} \\
             \sinh\frac{\mu c_i}{2}\tanh \frac{c_i}{2} & -\cosh\frac{\mu c_i}{2}
            \end{pmatrix}.
  \end{align*}
 In the last equality we used the substitution $\boldsymbol\mu = i\pi-\mu c_i$, coming from our hypothesis, and the formulae:
 $$\sinh(i \frac{\pi}{2}-x) = i\cosh(x) \;\;\text{and}\;\;\cosh(i \frac{\pi}{2}-x) = -i\sinh(x).$$
  If you consider the limit when $c_i \to 0$, we have that $$\sinh\frac{\mu c_i}{2}\coth \frac{c_i}{2} \to \mu$$ by using the Taylor expansion of the hyperbolic functions $\sinh$ and $\cosh$. This shows that $\mathbf{J}^{-1}\mathbf{T}_{\boldsymbol\mu}^{-1} \to J^{-1}T_\mu^{-1} = i\left(
   \begin{array}{cc}
   1 & -\mu \\
   0 & -1 \\
   \end{array}
   \right)$, as we wanted to prove. 
 
 In the general case in which $\e, \e' \in \{0,1,\infty\}$, we need to prove that $$\boldsymbol\Omega_\e(c_{i_1}, c_{i_2}, c_{i_3}) \to \Omega_{\e}.$$ This can be done by recalling the definition of $\Omega_{\e}$ and the fact that the geodesics $\Ax\left(A_\e\right)$ limits to $\e$ for $\e \in \{0, 1, \infty\}$. In fact, the map $${\boldsymbol\Omega}_{0} \co {\boldsymbol\Delta}_0(c_{i_3}, c_{i_1}, c_{i_2}) \to {\boldsymbol\Delta}_0(c_{i_1}, c_{i_2}, c_{i_3})$$ is the unique orientation-preserving map from the white hexagon ${\boldsymbol\Delta}_0(c_{i_3}, c_{i_1}, c_{i_2}) \subset {\boldsymbol\Delta}(c_{i_3}, c_{i_1}, c_{i_2})$ to the white hexagon ${\boldsymbol\Delta}_0(c_{i_1}, c_{i_2}, c_{i_3}) \subset {\boldsymbol\Delta}(c_{i_1}, c_{i_2}, c_{i_3})$ which sends $\Ax\left(A_{0}(c_{i_3}, c_{i_1}, c_{i_2})\right)$ to $\Ax\left(A_{\infty}(c_{i_1}, c_{i_2}, c_{i_3})\right)$. Hence it limits to the unique orientation preserving symmetry of $\Delta_0 \subset \Delta$ sending $0 \mapsto \infty$, that is $\Omega_0$. Similarly the unique orientation-preserving map $${\boldsymbol\Omega}_{1} \co {\boldsymbol\Delta}_0(c_{i_2}, c_{i_3}, c_{i_1}) \to {\boldsymbol\Delta}_0(c_{i_1}, c_{i_2}, c_{i_3})$$ which sends $\Ax\left(A_{1}(c_{i_2}, c_{i_3}, c_{i_1})\right)$ to $\Ax\left(A_{\infty}(c_{i_1}, c_{i_2}, c_{i_3})\right)$ limits to the unique orientation preserving symmetry of $\Delta_0 \subset \Delta$ sending $1 \mapsto \infty$, that is $\Omega_1$. This concludes the proof that ${\boldsymbol\Omega}_{\e} \to \Omega_{\e}$.
  \end{proof}
  
\section{Linear and BM--slices} \label{sec:BM}

In this section we discuss slices of quasi-Fuchsian space which contain the Kleinian groups defined by the $\uc$--gluing construction. The slices generalise the $c$--slice $\L_c(\Sigma_{1,1})$ of the quasi-Fuchsian space $\QF(\Sigma_{1,1})$ for the once punctured torus $\Sigma_{1,1}$ defined and studied by Komori and Yamashita in~\cite{kom_lin}. Given a pants decomposition $\PC = \{\s\}$ of $\Sigma_{1,1}$ and a positive number $c \in \R_+$, the $c$--\textit{slice} $\L_c$ is defined by
$$\L_{c} = \{\t \in \C_{(-\pi, \pi]} | (c,\t) \in \FN_{\C}\left(\QF(\Sigma_{1,1})\right)\}.$$
This slice has a special connected component, called the \textit{Bers--Maskit slice} $\BM_c(\Sigma)$, containing the Fuchsian points $\tau \in \R.$ This slice was studied by Komori and Parkkonen \cite{kom_ont}. See also McMullen \cite{mcm_com}. The geometric meaning of the Bers-Maskit slice is the following. Remember that every quasi-Fuchsian manifold contains a convex core whose boundary consists of locally convex embedded pleated surfaces. Let $\pl^\pm\co\QF(\Sigma) \to \mathcal{ML}(\Sigma)$ be the map which associates to a quasi-Fuchsian group the bending lamination in the top ($\pl^+$) or bottom ($\pl^-$) component of the boundary of the convex core, and let
$$X_c^\pm(\Sigma_{1,1}) = \{\t \in \L_c(\Sigma_{1,1})| \mathrm{supp}\left(\pl^\pm(G(c,\t))\right) = \s \}.$$

\begin{Theorem}[Theorem 2.1 of \cite{kom_ont}]\label{charac}
Let $\Sigma =\Sigma_{1,1}$ be a once punctured torus and let $\s \subset \Sigma$ be a simple closed curve. Then the complement of the Fuchsian locus $\F = \F_c(\Sigma)$ in $\BM_c = \BM_c(\Sigma)$ consists of two connected components $X_c^+ = X_{c,\s}^+(\Sigma_{1,1})$ and $X_c^- = X_c^-(\Sigma_{1,1})$ meeting $\F$ along the real line. The slice $\BM_c$ is simply connected and invariant under the action of the Dehn twist along $\s$. The component $X_c^-$ is in $\HH$, while $X_c^+$ is in the lower half plane $\LL$, and  they are interchanged by complex conjugation.
\end{Theorem}

\subsection{Naive definition of the slice $\L_{c}(\Sigma)$} \label{sub:definitions}

Given a surface $\Sigma$, fix a marking decomposition $(\PC, \DD) = (\s_{1},\ldots, \s_{\xi}, D_{1},\ldots ,D_{\xi})$ on it, see Section \ref{sec:marking}. Let $\FN_{\C}\co\QF(\Sigma) \to (\C_+/2i\pi)^\xi \times (\C/2i\pi)^{\xi}$ be the complex Fenchel--Nielsen embedding of $\QF(\Sigma)$ into $(\C_+/2i\pi)^\xi \times (\C/2i\pi)^{\xi}$ associated to $(\PC, \DD)$ and defined by $$\FN_\C(G) = (\frac{\l_1}{2}, \ldots, \frac{\l_{\xi}}{2}, \t_1, \ldots, \t_\xi),$$ where $\l_i$ is the complex length of $\s_i$ and $\t_i$ is the complex twist, see Section \ref{sub:fenchel_nielsen_coordinates}. We define the $\uc$--\textit{slice} $\L_\uc$ to be the set
$$ \L_{\uc} = \L_{\uc}(\Sigma) = \{\ut \in (\C_{(-\pi, \pi)})^\xi | (\uc, \ut) \in \FN_{\C}\left(\QF(\Sigma)\right)\}.$$

First, the $\uc$--slice has some periodicities and symmetries.
\begin{Proposition}
In the $\uc$--slice $\L_{\uc}(\Sigma)$ we have:
\begin{enumerate}[(i)]
\item $(\uc, \ut) \in \L_{c} \Rightarrow (\uc, \ut + 2\uc \cdot \underline{k}) \in \L_{c},$
\item $(\uc, \ut) \in \L_{c} \Rightarrow (\uc, -\ut) \in \L_{c},$
\end{enumerate}
where $\underline{k} \in \Z^\xi$.
\end{Proposition}

\begin{proof}
Since we can consider a different marking (in particular a different set of dual curves $\DD$), we have the periodicity $\t_i \mapsto \t_i + 2 c_i$ given by exchanging the dual curve $D_i$ with the curve $T_{\s_i} D_i$ obtained by doing a Dehn--twist of the curve $D_i$ around the pants curve $\s_i$. This proves (i).

Similarly, by exchanging the dual curve $D_i$ with its inverse $D_i^{-1}$, we observe the symmetry $\t_i \mapsto -\t_i$ with respect to the hyperplane $\Re \t_i = 0$. Hence we have statement (ii).
\end{proof}

Second, since $\FN_{\R} \co \F \to \R_{+}^\xi \times \R^\xi$ is a bijection, we have also the following result:

\begin{Proposition}
The slice $\L_{c}$ has a connected component containing the real locus $\R^\xi$.
\end{Proposition}

This proposition allows us to define the Bers--Maskit slice $\BM_c(\Sigma)$. 

\begin{Definition}
The \textit{Bers--Maskit slice} $\BM_c(\Sigma)$ is the connected component of $\L_c(\Sigma)$ which contains the Fuchsian locus $\F_c(\Sigma)$, that is the points $\ut \in \R^\xi \cap \L_{c}.$
\end{Definition}

It would be interesting to find the analog of Therem \ref{charac}. Note that the fact that the cardinality of $\PC$ is greater than one makes this situation more subtle since one can bend along different pants curves in opposite directions and this implies that the surfaces obtained would not be convex. Note also that the slice $\L_{\uc}$ is a generalisation of the \textit{shearing plane} $\E_{\PC,\uc}$ of Series \cite{ser_ker}. One could also define the linear slice as a generalisation of the \textit{horoplane} $\H_{\PC,c}$, where $c \in \R_{+}$. In that case, 
$$ \L'_{c} = \{(\uc,\ut) \in \R_+^\xi \times {\C_{[0, \pi)}}^\xi\cap \FN_{\C}\left(\QF(\Sigma)\right)| \sum_{i=1}^{\xi}c_i =c\}.$$ As discussed in \cite[Theorem 7.3]{ser_ker} this definition for $\L_{\uc}(\Sigma)$ might be too rigid. McMullen discussion on complex earthquakes and earthquake disks might also help in understanding the `right' definition for these slices.

\subsection{Connected components of $\L_c(\Sigma)$} \label{sub:extra_components}

\begin{figure}
[hbt] \centering
\includegraphics[height=3.5cm]{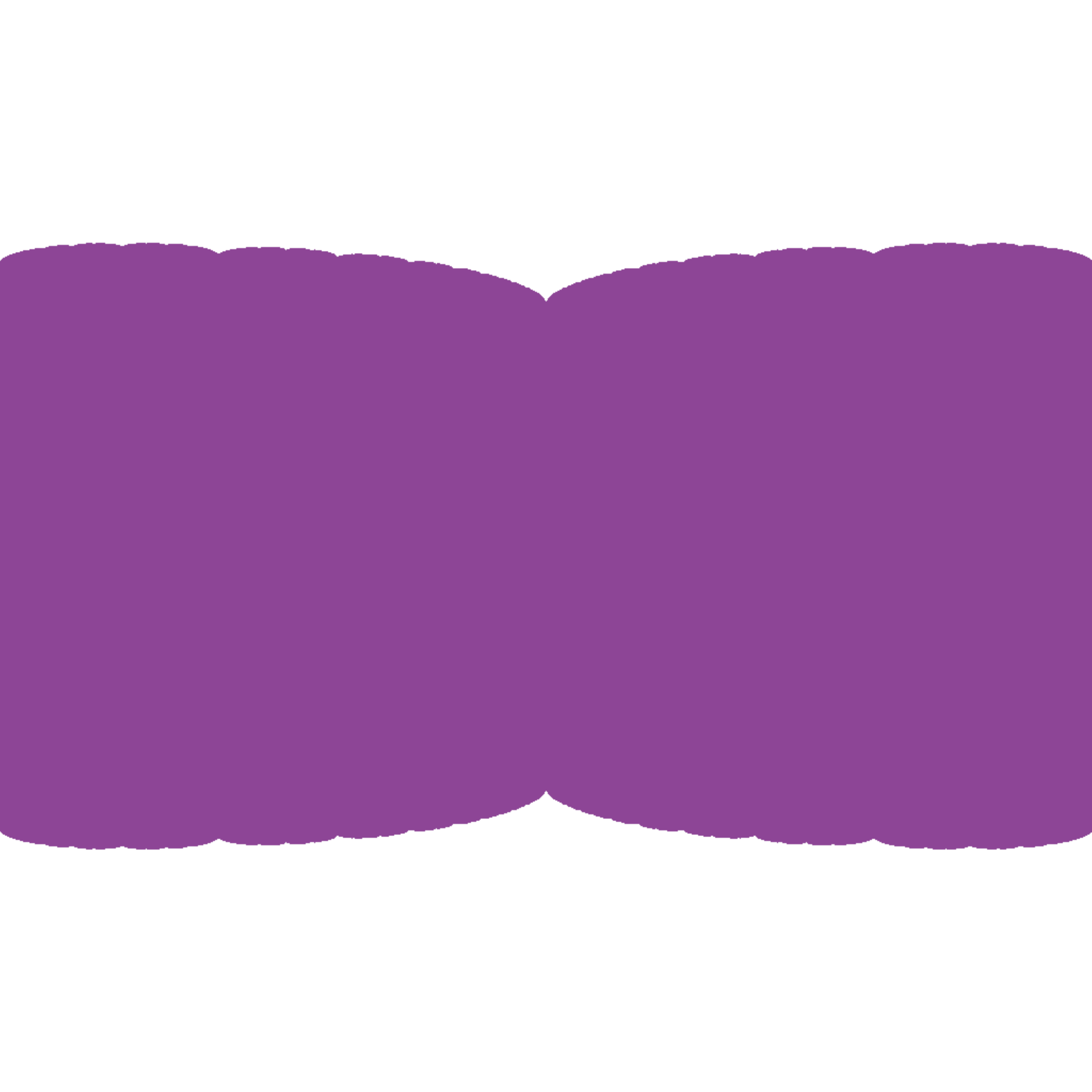}\;\;\;\; \includegraphics[height=3.5cm]{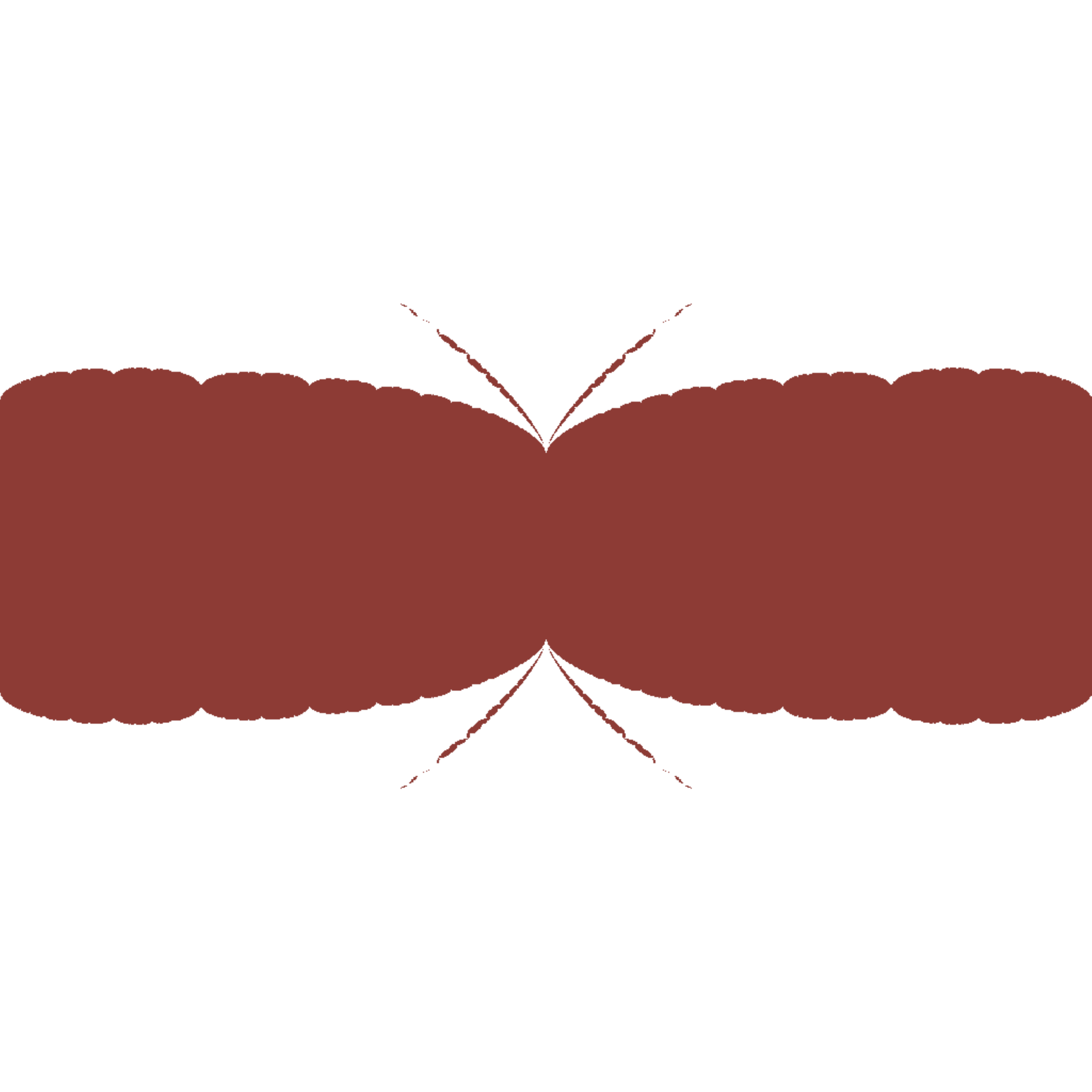}\\ 
\includegraphics[height=3.5cm]{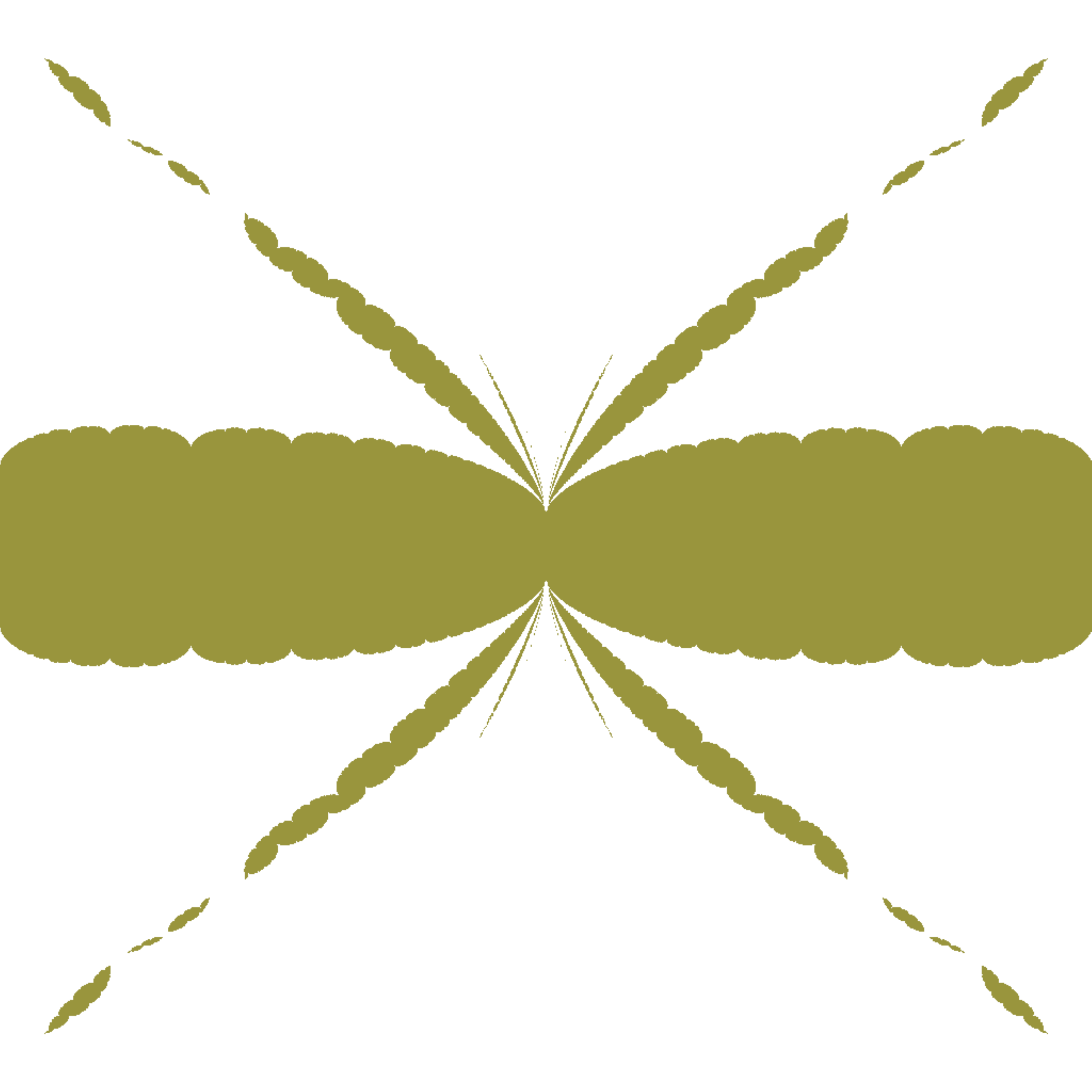} \;\;\;\;\includegraphics[height=3.5cm]{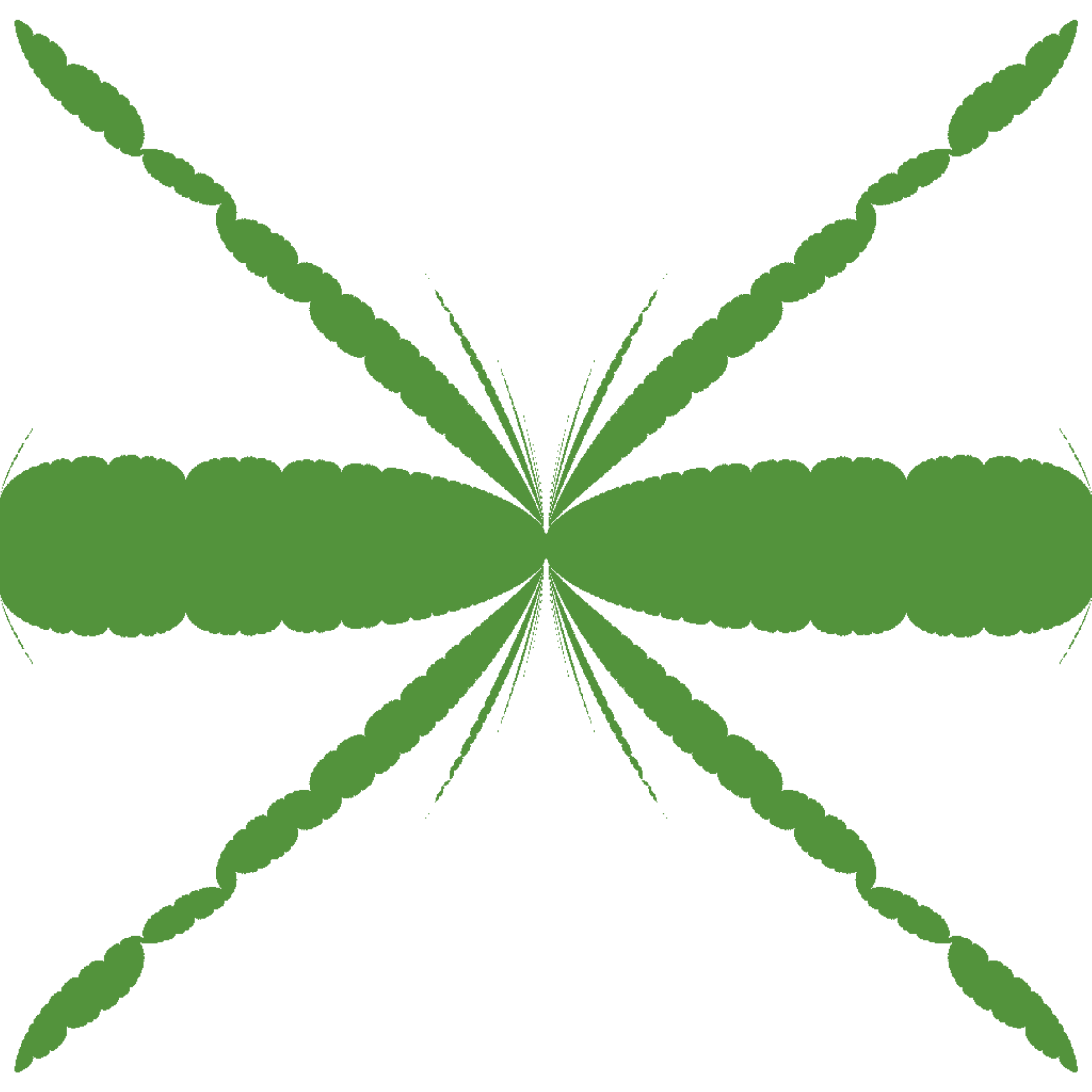}
\includegraphics[height=4cm]{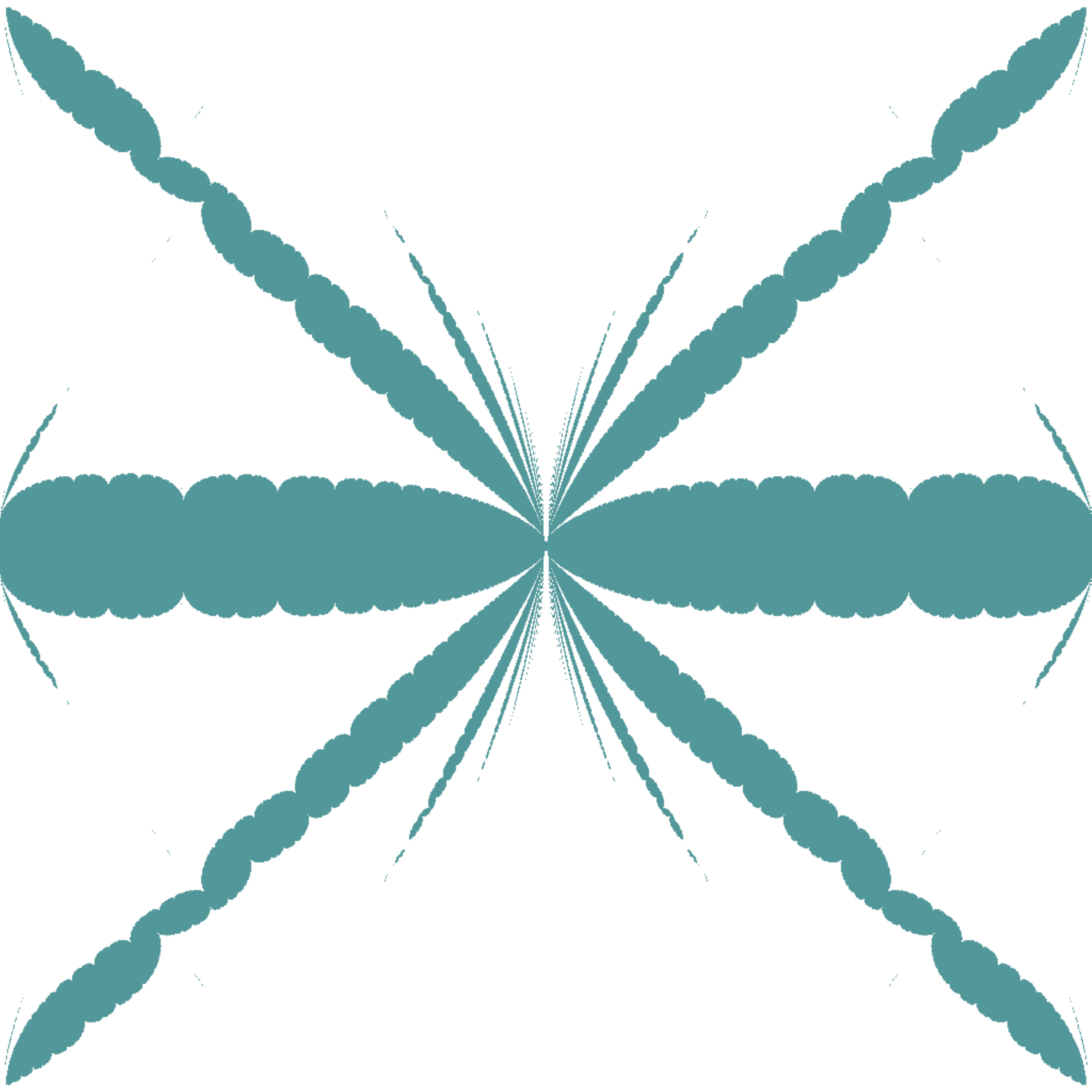}\;\;\;\;\includegraphics[height=3.5cm]{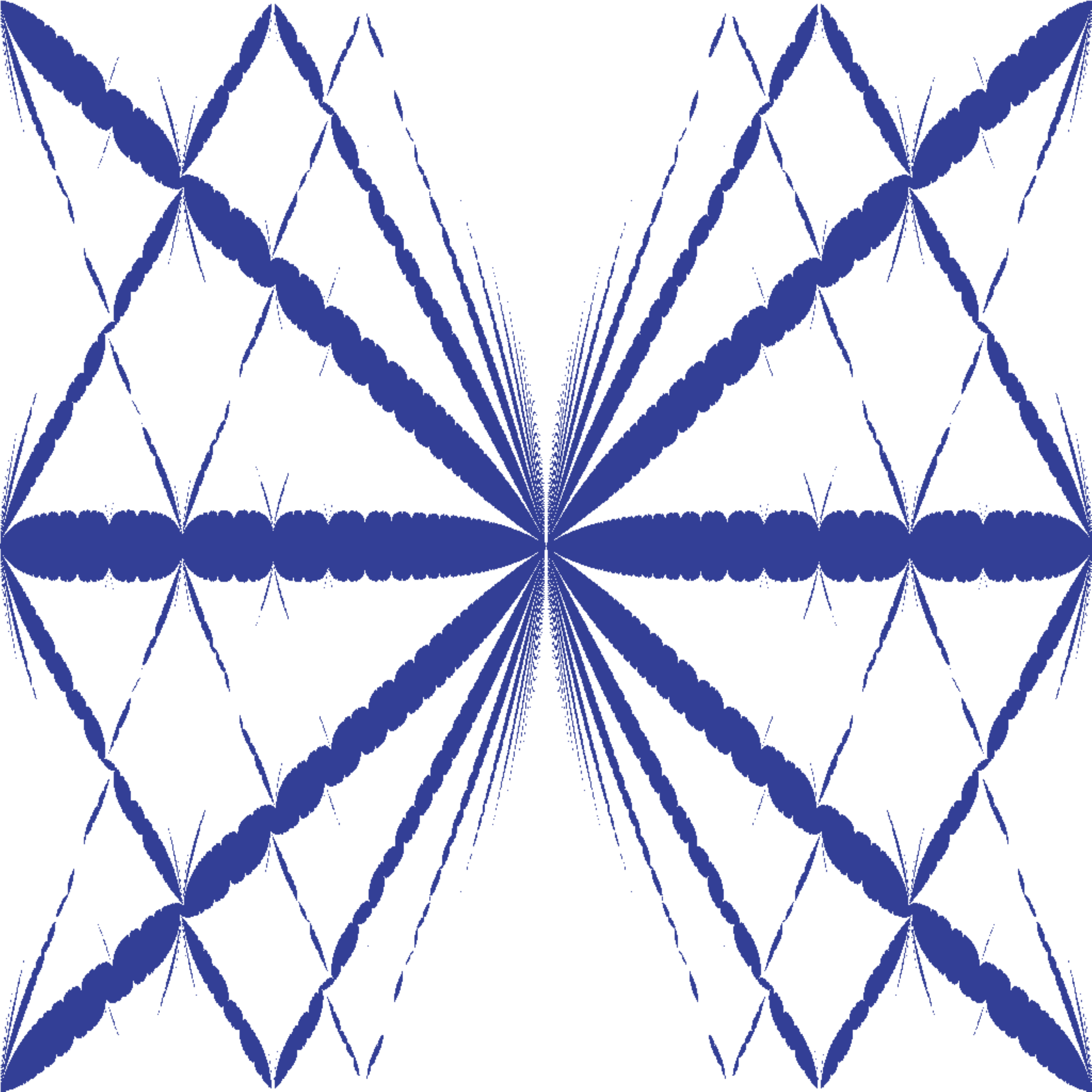}\;\;\;\;\includegraphics[height=3.5cm]{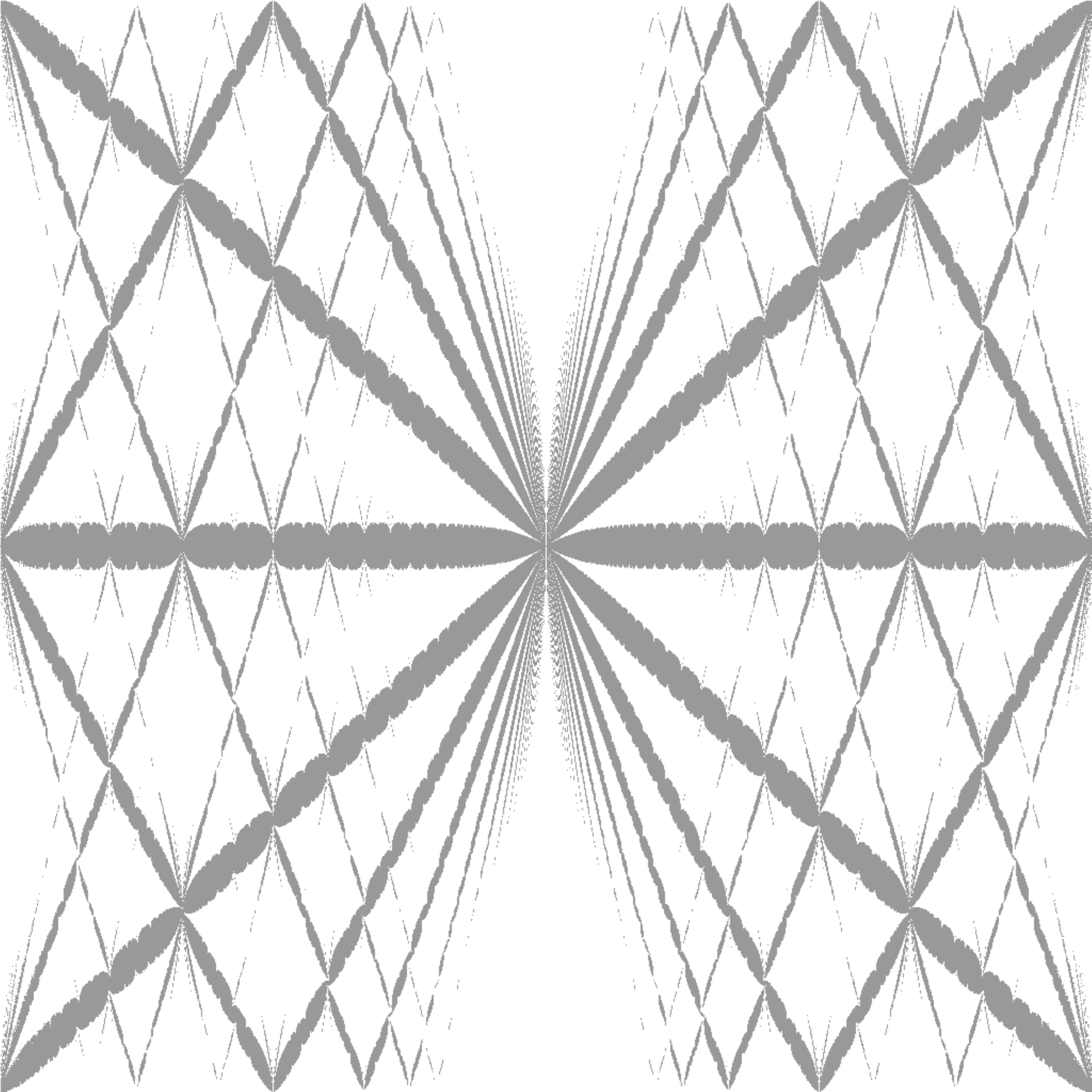}
\caption{The linear slice $\L_c$ when $c = 1,2,3,4,5,10, 20$ (joint work with Y. Yamashita)}
\label{conn}
\end{figure}

\begin{figure}
[hbt] \centering
\includegraphics[height=3.5cm]{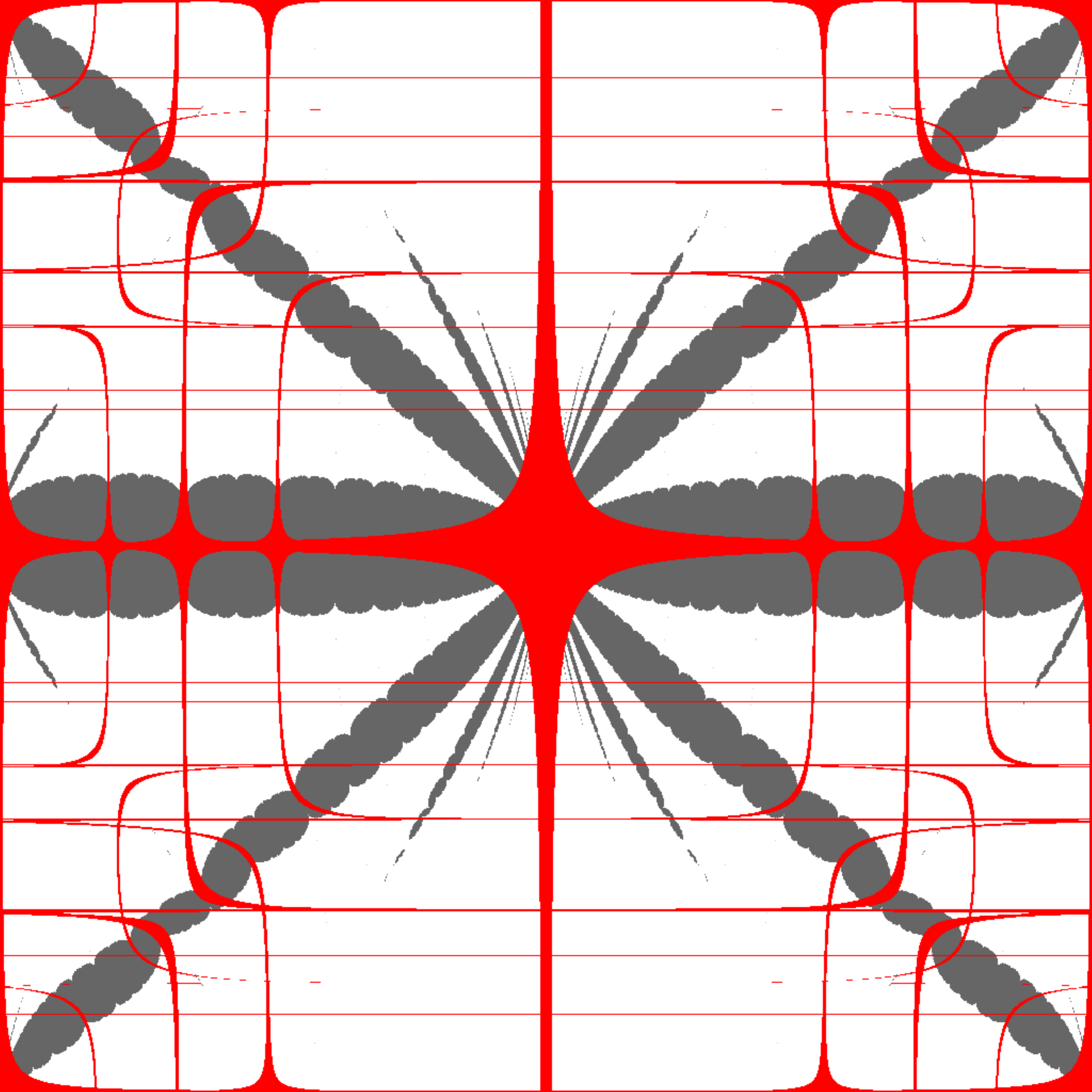}\;\;\;\; \includegraphics[height=3.5cm]{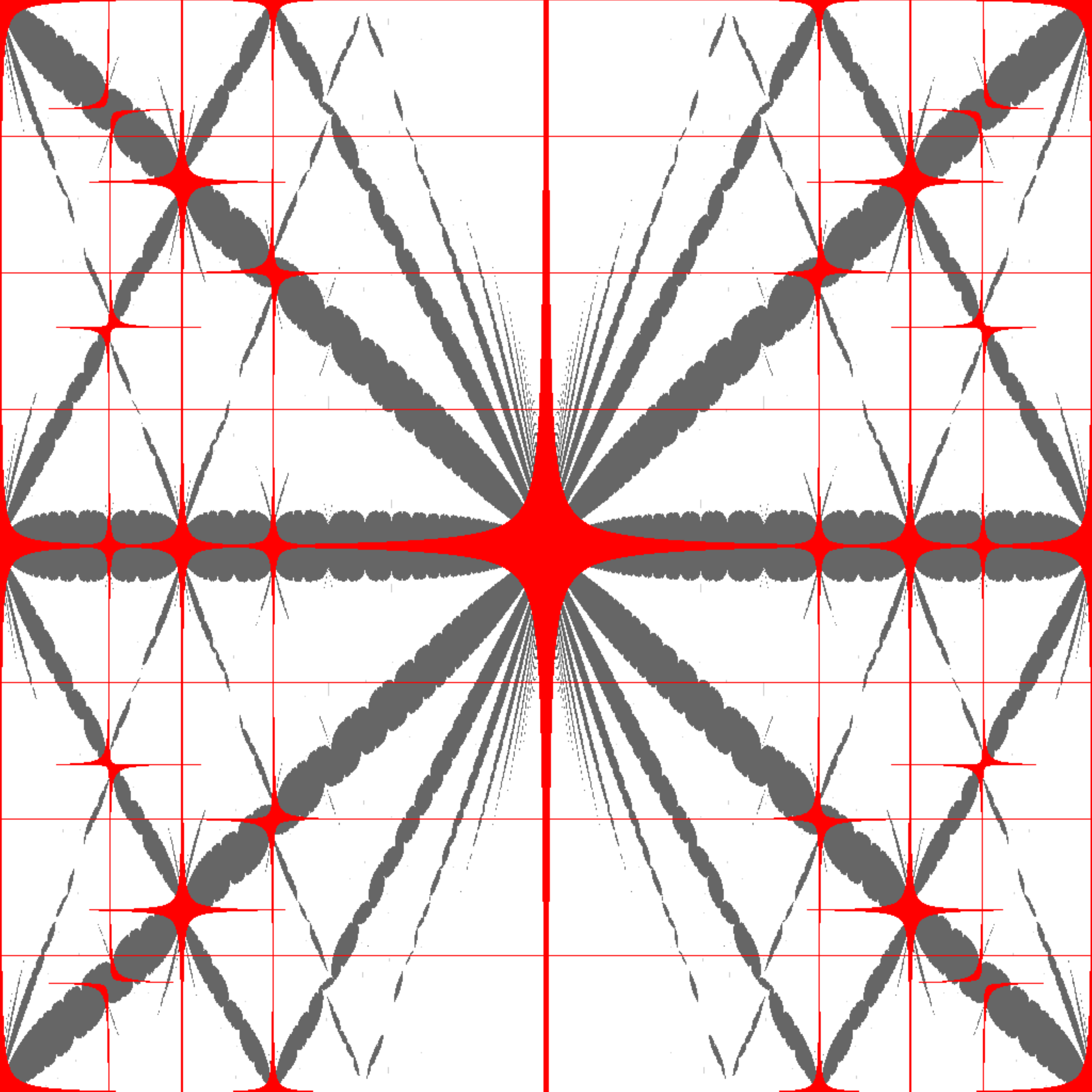}
\caption{The linear slice $\L_c$ with in red the `undecidable locus' (joint work with Y. Yamashita)}
\label{conn3}
\end{figure}

\begin{figure}
[hbt] \centering
\includegraphics[height=6cm]{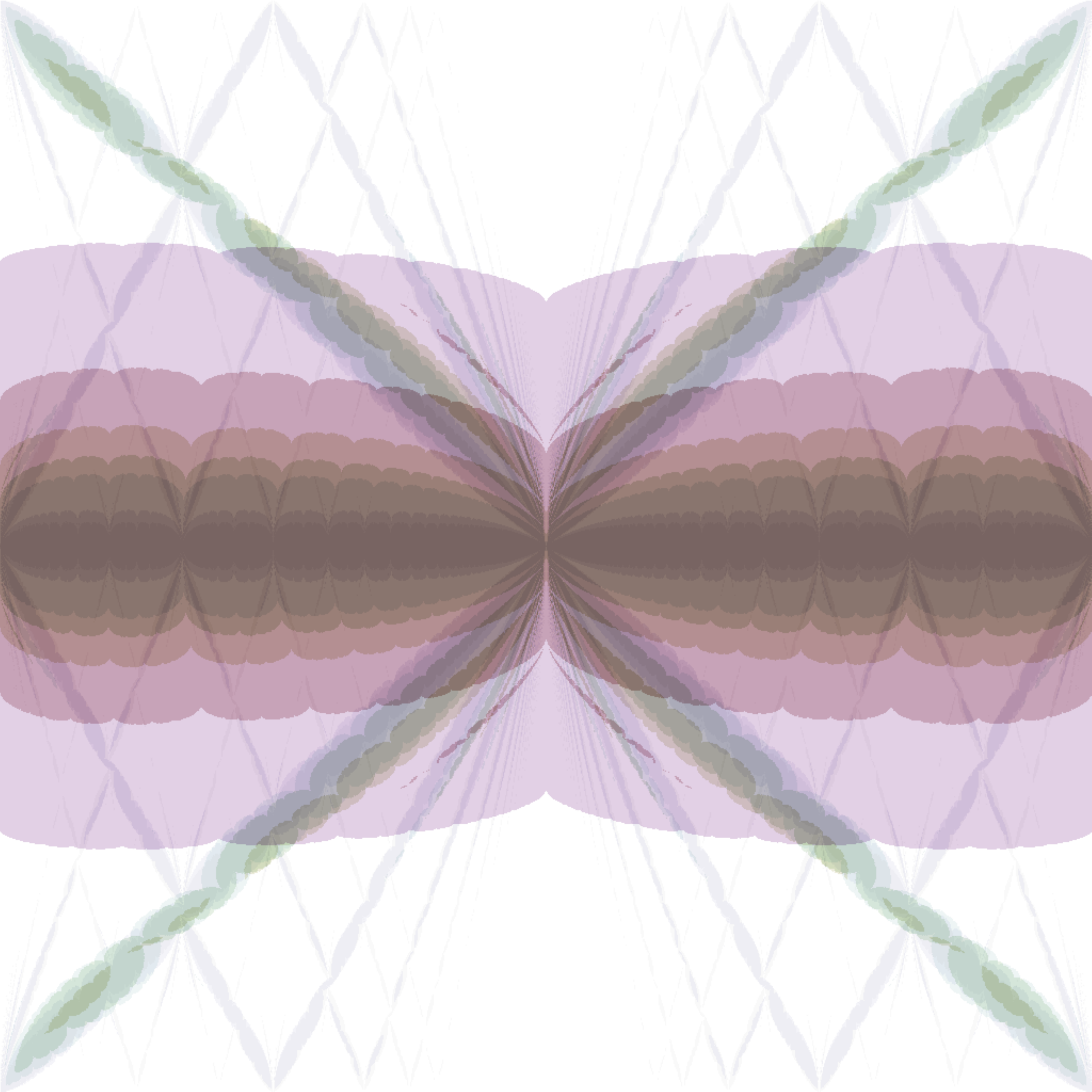}
\caption{The overlapping of the linear slices $\L_c$ when $c = 1,2,3,4,5,10, 20$ (joint work with Y. Yamashita).}
\label{conn2}
\end{figure}

Komori and Yamashita~\cite{kom_lin} proved that there exist two (positive) real constants $0 < C_0 \leqslant C_1$ such that for any $0< c < C_0$ the linear slice $\L_c(\Sigma_{1,1})$ is connected (and hence coincides with $\BM_c(\Sigma_{1,1})$), while for all $c > C_1$ the linear slice is not connected. They also conjectured:
\begin{Conjecture}
$C_0 = C_1.$
\end{Conjecture}

Together with Yamashita, we wrote a computer program which draws these slices $\L_c$ for different values of the parameter $c$, see Figure \ref{conn} and \ref{conn2}. It is worth noting that this program assumes Bowditch's conjecture is correct; see Conjecture A of \cite{bow_mar}. These figures suggest many questions and conjecture about the bahaviour of these slices. For example, Figure \ref{conn2} shows evidence for the following conjecture. Recall that $\M^{tot}$ is the total Maskit slice (defined with respect to the pants decomposition $\PC$).

\begin{Conjecture}\label{union}
Given $c_1, c_2 \in \R_+$, we have:
\begin{enumerate}
\item If $c_1 \leqslant c_2,$ then $\BM_{c_2} \subseteq \BM_{c_1}$.
\item $\M^{tot}$ corresponds to $\cup_{c > 0} \BM_{c}.$
\end{enumerate}
\end{Conjecture}

Komori and Yamashita's proofs cover only the case of the once punctured torus $\Sigma_{1,1}$, but we think that these `exotic components' are related to the wrapping maps, see \cite{and_alg, bro_sel}, and the obstruction of the existence of these wrapping maps as studied by Evans and Holt \cite{eva_non}. 

\begin{appendix}
  
\section{Calculation for the standard pair of pants}\label{app:degenerate}

In Section 5.2 of \cite{mas_mat} Maskit describes a group $G'$ such that $\HH^2/G'$ is a three holed sphere with infinite funnel ends such that the length of the three geodesic parallel to the boundary is $2c_{1}, 2c_{2}$ and $2c_{3}$, respectively. In particular, $G'$ is generated by $A'_{\infty}$ and $A'_{0}$ (or, alternatively, the group $G'$ has the following presentation $G' = \langle A'_\infty, A'_0, A'_1 | A'_\infty A'_0A'_1 = \Id \rangle$), where $$A'_\infty = A'_\infty(c_{1},c_{2},c_{3}) =\left(
\begin{array}{cc}
e^{c_{1}} & 0 \\
0 & e^{-c_{1}} \\
\end{array}
\right),$$ $$A'_0 = A'_0(c_{1},c_{2},c_{3}) = \frac{1}{\sinh\nu_1} \left(
\begin{array}{cc}
\sinh(\nu_1 - c_{2}) & \sinh c_{2} \\
-\sinh c_{2} & \sinh(\nu_1 + c_{2}) \\
\end{array}
\right),$$ $$A'_1 = A'_1(c_{1},c_{2},c_{3}) = \frac{1}{\sinh\nu_2} \left(
\begin{array}{cc}
\sinh(\nu_2 - c_{3}) & e^{c_{1}}\sinh c_{3} \\
-e^{-c_{1}}\sinh c_{3} & \sinh(\nu_2 + c_{3}) \\
\end{array}
\right)$$
and $\nu_1$ and $\nu_2$ are positive numbers defined by Equation \ref{nu}.

As explained before, when you have to take limits as $c_{i_j}\to 0$, it is important to choose well the conjugation class of the subgroup of $\PSL(2,\C)$. So, inspired by Parker--Parkonnen \cite{par_coo}, we conjugate this group by the matrix $V =  \left(
\begin{array}{cc}
\cosh(\frac{c_{1}}{2}) & -\cosh(\frac{c_{1}}{2}) \\
\sinh(\frac{c_{1}}{2}) & \sinh(\frac{c_{1}}{2}) \\
\end{array}
\right)$ which maps $$0 \mapsto -\coth(\frac{c_{1}}{2}), \;\;\infty \mapsto \coth(\frac{c_{1}}{2})\;\;\text{and}\;\; 1 \mapsto 0.$$
In this way, we get the group $\boldsymbol\Gamma(c_{1}, c_{2}, c_{3}) = \langle A_\infty, A_0, A_1 |  A_\infty A_0 A_1 = \Id  \rangle$ introduced in Section \ref{sec:std},  where $A_\e = VA'_\e V^{-1}$ for all $\e$ in the cyclically ordered set $\{0,1,\infty\}$.

If one or more of the boundary components are punctures, then the calculations are different. In the case $c_1, c_2 \neq 0$ and $c_3= 0$, we have $G' = \langle A'_\infty, A'_0 \rangle$, where 
$$A'_\infty = A'_\infty(c_{1},c_{2}, 0) =\left(
\begin{array}{cc}
e^{c_{1}} & 0 \\
0 & e^{-c_{1}} \\
\end{array}
\right),$$ 
$$A'_0 = A'_0(c_{1},c_{2}, 0) = \frac{1}{\sinh\nu} \left(
\begin{array}{cc}
\sinh(\nu - c_{2}) & \sinh c_{2} \\
-\sinh c_{2} & \sinh(\nu + c_{2}) \\
\end{array}
\right)$$ 
and $\coth \nu = \frac{\cosh c_{1}\cosh c_{2}+\cosh c_{3}}{\sinh c_{1}\sinh c_{2}}$ and $\nu >0$. Conjugating by the map $V$ above we get $\boldsymbol\Gamma(c_{1}, c_{2}, 0) = \langle A_\infty, A_0 \rangle$, where
$$A_\infty= A_\infty(c_{1},c_{2},0) = \left(
\begin{array}{cc}
\cosh c_{1}  & \cosh c_{1}+1 \\
\cosh c_{1}-1 & \cosh c_{1} \\
\end{array}
\right),$$ $$A_0 = A_0(c_{i_1},c_{i_2},c_{i_3}) =  \left(
\begin{array}{cc}
\cosh c_{2} & - \frac{\cosh c_1 +1}{\cosh c_1}\tanh(\frac{\nu}{2}) \sinh c_{2} \\
-\frac{\cosh c_1 -1}{\cosh c_1}\coth(\frac{\nu}{2}) \sinh c_{2} & \cosh c_{2} \\
\end{array}
\right).$$

In the case $c_1\neq 0$ and $c_2 = c_3= 0$, we have 
$\langle A'_\infty, A'_0\rangle$, where 
$$A'_\infty = A'_\infty(c_{1},0, 0) =\left(
\begin{array}{cc}
e^{c_{1}} & 0 \\
0 & e^{-c_{1}} \\
\end{array}
\right),$$ 
$$A'_0 = A'_0(c_{1},0, 0) = \left(
\begin{array}{cc}
1 +\beta & -\beta \\
\beta &  1-\beta \\
\end{array}
\right),$$
where $\beta = - \frac{\cosh c_1 +1}{\sinh c_1}.$
After conjugating by the map $V$ above we get $\boldsymbol\Gamma(c_{1}, 0, 0) = \langle A_\infty, A_0 \rangle$, where
$$A_\infty= A_\infty(c_{1},0,0) = \left(
\begin{array}{cc}
\cosh c_{1}  & \cosh c_{1}+1 \\
\cosh c_{1}-1 & \cosh c_{1} \\
\end{array}
\right),$$ $$A_0 = A_0(c_{i_1},0, 0) =  \left(
\begin{array}{cc}
1 & 0 \\
-2  & 1 \\
\end{array}
\right).$$

\section{Examples: the case of the once punctured torus and our holed sphere}\label{app:example}

\subsection{Once punctured torus: the HNN--extension}\label{ssub:once}

When $\Sigma = \Sigma_{1,1}$ is a once-punctured torus, our construction is based on the description of Parker and Parkkonen~\cite{par_coo}. Let $X$ be a hyperbolic structure $X \in \teich(\Sigma)$ on $\Sigma$ and let $\s \subset \Sigma$ be a simple closed geodesic. Let $X_{0}$ be the hyperbolic structure on a once punctured cylinder obtained by cutting $X$ along $\s$. The two boundary components have the same length, say $2c \in \R_{+}$. 

\begin{figure}
[hbt] \centering
\includegraphics[height=4cm]{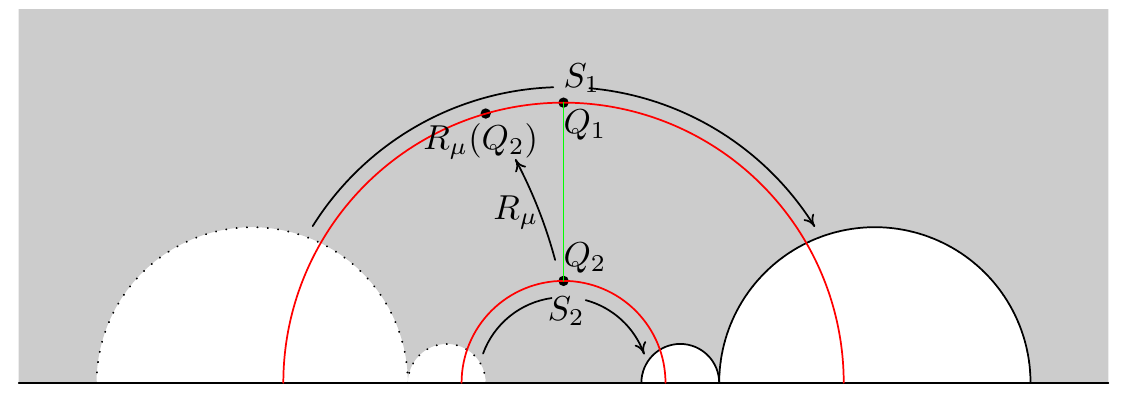}
\caption{The fundamental domain $F$ for the group $G_{0}$, in grey.}
\label{fig5}
\end{figure}

We can realise $X_{0}$ as a quotient $X_{0} = \HH^2/G_{0}$, where $G_{0} = \langle S_1, S_2 \rangle$, where:
$$S_1 = \left(
\begin{array}{cc}
\cosh c & \cosh c +1 \\
\cosh c - 1 & \cosh c \\
\end{array}
\right),\; S_2= \left(
\begin{array}{cc}
\cosh c & \cosh c - 1 \\
\cosh c + 1 & \cosh c \\
\end{array}
\right).$$

The fixed points of these transformations are $\mathrm{Fix}^\pm(S_1) =\pm \coth \frac{c}{2}$ and $\mathrm{Fix}^\pm(S_2) =\pm \tanh \frac{c}{2}$ and the products $$K_{1}=S_2^{-1}S_1 = \left(
\begin{array}{cc}
2\cosh c -1 & 2\cosh c \\
-2\cosh c & -2\cosh c-1 \\
\end{array}
\right)\;\; \mathrm{and}\;\; K_{2}=S_1S_2^{-1} = \left(
\begin{array}{cc}
-2\cosh c -1 & 2\cosh c \\
-2\cosh c & 2\cosh c-1 \\
\end{array}
\right)$$ are parabolic elements with fixed points at $-1$ and $1$, respectively. The two hyperbolic transformations $S_1$ and $S_2$ correspond to the two boundary geodesics of $X_{0}$, while $K_{1}$ and $K_{2}$ correspond to the puncture. A fundamental domain $F$ for the action of $G_0 = \langle S_1, S_2 \rangle$ on $\HH^{2}$ is:
\begin{equation}\label{fund_once}
 F = \{z \in \HH^{2} | |z \pm C_1| > R_1, |z \pm C_2|> R_2 \}, 
\end{equation}
where $C_1 = \frac{\cosh c}{\cosh c -1}$, $R_1= \frac{1}{\cosh c -1}$, $C_2 = \frac{\cosh c}{\cosh c +1}$ and $R_2 = \frac{1}{\cosh c + 1}$. The circles with center $\pm C_1$ and radius $R_1$ are the isometric circles of $S_1$ and $S_1^{- 1}$, while the ones with center $\pm C_2$ and radius $R_2$ are the isometric circles of $S_2$ and $S_2^{- 1}$. See Figure \ref{fig5}.

Notice that there is a unique simple geodesic arc on $X_{0}$ perpendicular to both the two geodesic boundary curves. A distinguished lift of this arc to the fundamental set $F$ is the intersection of $F$ with the positive imaginary axis, that is, the segment of the positive imaginary axis connecting $Q_2 = i \tanh \frac{c}{2} \in \mathrm{Axis}(S_2)$ to $Q_1 = i\coth \frac{c}{2}\in \mathrm{Axis}(S_1).$ The hyperbolic surface $X$ can be reconstructed by gluing together the two geodesic boundary components of $\Sigma_{0}$. This is done by adding to the group $G_{0}$ a hyperbolic M\"obius transformation $R_{\mathbf{\mu}}$. Since the transformation $R_{\mathbf{\mu}}$ has to conjugate the cyclic subgroups $\langle S_1\rangle$ and $\langle S_2\rangle$ in a way compatible with the gluing operation, $R_{\mathbf{\mu}}$ has to satisfy the relationship $R_{\mathbf{\mu}} S_2 R_{\mathbf{\mu}}^{-1}= S_1$, so it is fixed up to the parameter $\mathbf{\mu} \in \R$:
$$R_{\mathbf{\mu}} = \left(
\begin{array}{cc}
\cosh \frac{\mathbf{\mu}}{2}\coth \frac{c}{2} & -\sinh\frac{\mathbf{\mu}}{2} \\
-\sinh\frac{\mathbf{\mu}}{2} & \cosh\frac{\mathbf{\mu}}{2}\tanh \frac{c}{2} \\
\end{array}
\right)\in \SL(2, \R).$$
We form, in this way, a new Fuchsian group $G$ which is a \textit{Higman--Neumann--Neumann--extension} (shortened to HNN--extension) of $G_{0}$ by the element $R_{\mathbf{\mu}}$:
$$G = (G_{0})*_{\langle R_\t \rangle} = \langle G_{0}, R_\t \rangle.$$

There exists exactly one parameter $\mathbf{\mu}_{0} \in \R$ which reconstructs the original (marked) hyperbolic structure $X$ on the surface $\Sigma$, but the group $G$ is a Fuchsian group for any real parameter $\mathbf{\mu} \in \R$. (In fact, this description corresponds to the Fenchel-Nielsen construction for Fuchsian groups.) In addition this parameter has also the following geometrical interpretation: the transformation $R_{\mathbf{\mu}}$ maps the point $Q_2 = i \tanh \frac{c}{2} \in \mathrm{Axis}(S_2)$ to a point $Q_3$ on the axis of $S_1$, namely
$$Q_3 = R_{\mathbf{\mu}}(Q_2) = i\coth \frac{c}{2}\left(\sech\mathbf{\mu}+i\tanh\mathbf{\mu}\right) \in \mathrm{Axis}(S_1).$$ 
The parameter $\mathbf{\mu} \in \R$ is the signed distance of this point $Q_3$ from $Q_1 = i\coth \frac{c}{2}$, where the sign of $\Re\mathbf{\mu}$ is chosen to be positive if, moving from $Q_1$ to $Q_3$, takes one in a positive (anti-clockwise) direction around the axis of $S_1$, see Figure \ref{fig5}. 

  The construction makes sense also when one considers complex parameters $\mathbf{\mu} \in \C_{(-\pi, \pi)}$, but it does not give a group in $\QF$ for all $\mathbf{\mu} \in \C_{(-\pi, \pi)}$. However, by Theorem 3.5 of \cite{ser_onk}, then there exists $\epsilon > 0$, such that, if $|\mathbf{\mu}| < \e$, then $G = G(c, \mathbf{\mu})$ is quasi-Fuchsian.

\begin{Remark}\label{C_once}
  If you consider the action of $G$ on $\C \subset \hat\C$ (when $\mathbf{\mu} \in \C_{(-\pi, \pi)}$), you can see that the circle in $\C$ with centre at $i \tanh(\frac{c}{2}) \tan(\frac{\Im \mathbf{\mu}}{2})$ and radius $\tanh(\frac{c}{2}) \mathrm{sec}(\frac{\Im \mathbf{\mu}}{2})$ is mapped by $R_{\mathbf{\mu}}$ to the circle in $\C$ with centre at $-i \coth(\frac{c}{2}) \tan(\frac{\Im \mathbf{\mu}}{2})$ and radius $\coth(\frac{c}{2}) \mathrm{sec}(\frac{\Im \mathbf{\mu}}{2})$, where $\Im \mathbf{\mu} \in (-\pi, \pi)$. Moreover these circles are mapped to themselves under $\langle S_2 \rangle$ and $\langle S_1 \rangle$ respectively, since they pass through the fixed points of $S_2$ and $S_1$, respectively. These circles are hypercycles around $\Ax(S_1)$ and $\Ax(S_2)$, see Section \ref{sec:cglu} for the definition of hypercycle.
\end{Remark}

\subsection{Four holed sphere: the AFP--construction}\label{ssub:four}

We now discuss briefly the case of the four holed sphere $\Sigma = \Sigma_{0,4}$, see also Komori \cite{kom_deg}. Fix a hyperbolic structure $X \in \teich(\Sigma)$ on $\Sigma$ and let $\s \subset \Sigma$ be a simple closed geodesic. Let $X_{1}$ and $X_2$ be the hyperbolic structures on the doubly-punctured disks (obtained by cutting $X$ along $\s$), where the two boundary components $\s_1 \subset \partial X_{1}$ and $\s_2 \subset \partial X_{2}$ have the same length, say $4c \in \R_{+}$. 

We can realise $X_{1}$ as a quotient $X_{1} = \HH^2/G_{1}$, where $G_{1} = <V_{1}, V_{2}>$, where 
$$V_{1} = \left(
\begin{array}{cc}
\cosh c +1 & +\cosh c \\
-\cosh c & -\cosh c +1\\
\end{array}
\right),\; V_{2}= \left(
\begin{array}{cc}
\cosh c-1 & -\cosh c \\
+\cosh c & -\cosh c-1 \\
\end{array}
\right).$$

\begin{figure}
[hbt] \centering
\includegraphics[height=4cm]{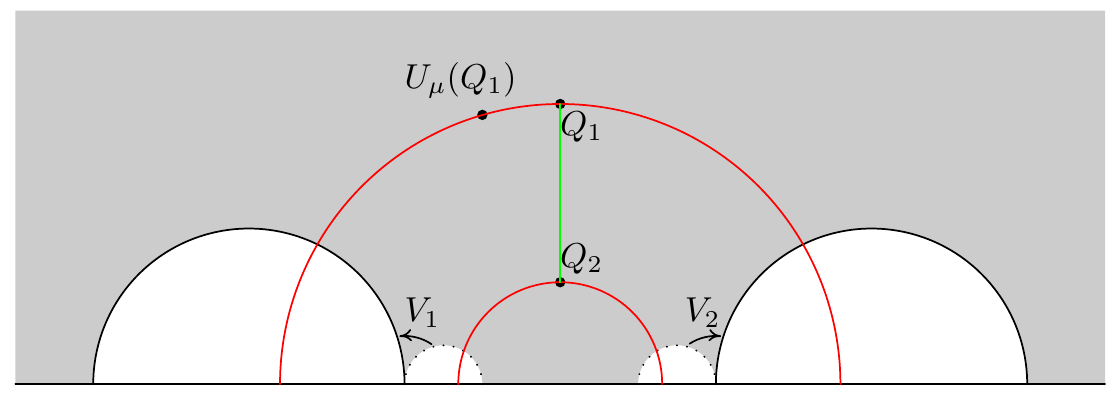}
\caption{The fundamental set $F'$ for the group $G_{1}$.}
\label{fig6}
\end{figure}

\begin{figure}
[hbt] \centering
\includegraphics[height=4cm]{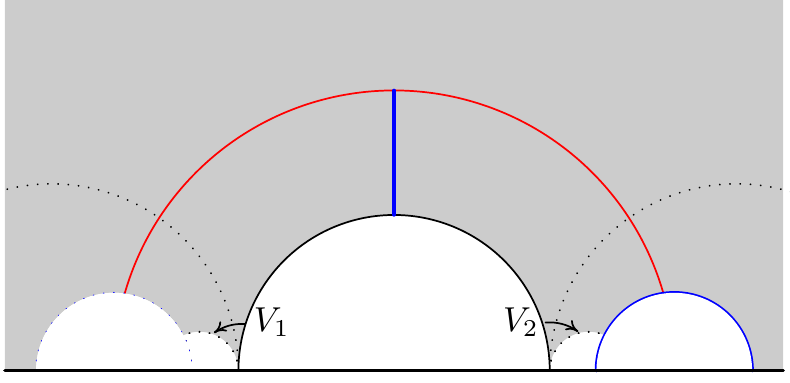}
\caption{The fundamental set $\hat{F}$ for the group $G_{1}$.}
\label{fig66}
\end{figure}

Using the notation of the previous section, we have that $V_{1}V_{2} = S_1^{-2}$ is a hyperbolic transformation with multiplier $2c \in \R_{+}$. The transformations $V_{1}$ and $V_{2}$ have fixed points $\mathrm{Fix}(V_{1}) = -1$ and $\mathrm{Fix}(V_{2}) = 1$ and correspond to the two punctures, while the hyperbolic transformations $V_{1}V_{2} = S_1^{-2}$ and $V_{2}V_{1} = S_2^{-2}$ correspond to the boundary geodesic of $X_{1}$. A fundamental set $F'$ for the action of the group $G_{1}$ on $\HH^2$ is defined by:
\begin{equation}\label{fund_four_4}
 F' = \{z \in \HH^{2} | |z \pm C_1| \ge R_1, |z \pm C_2|> R_2 \}, 
\end{equation} where $C_1, C_2, R_1, R_2$ are defined in Section \ref{ssub:once}; see Figure \ref{fig6}. Note that the geodesic along which we do the gluing, in this case, has double length with respect to the previous case. You can see this easily when you write down a different fundamental set for the action of this group on $\HH^2$:
\begin{equation}\label{F1}
  \hat{F}  = \{z \in \HH^{2} | |z| \ge 1, |z \pm C_1'| > R_1', |z \pm C'_2| \ge R'_2 \},
\end{equation}
where $C_1' = \frac{2\cosh c}{2\cosh c -1}$, $R_1' = \frac{1}{2\cosh c -1}$, $C'_2 = \pm\frac{1-2\cosh^2 c}{2\cosh c (\cosh c -1)}$ and $R'_2 = \frac{1}{2\cosh c (\cosh c -1)}$, see Figure \ref{fig66}. The circles with centers $\pm C_1$ and radius $R_1$ are the images of the unit circle under $V_1^{-1}$ and $V_2^{-1}$, while the circles with centers $\pm C'_1$ and radius $R'_1$ are the isometric circles of $S_1^2$ and $S_1^{-2}$. 

In this case, also, there is a unique simple geodesic arc on $X_{1}$ connecting the geodesic boundary curve to itself and intersecting it perpendicularly. A distinguished lift of this arc to the fundamental set $F'$ is the intersection of $F'$ with the positive imaginary axis, that is the segment of the positive imaginary axis connecting $Q_2 = i \tanh \frac{c}{2} \in \mathrm{Axis}(S_2)$ to $Q_1 = i\coth \frac{c}{2}\in \mathrm{Axis}(S_1).$

Similarly, we can realise $X_{2}$ as a quotient $X_{2} = \HH^2/G_{2}$, where $G_{2}$ is a Fuchsian group obtained by conjugating the group $G_{1}$ by a hyperbolic transformation $U_{\mathbf{\mu}}$, with $\mathbf{\mu} \in \R$: $$G_{2} = U_{\mathbf{\mu}} G_{1}U_{\mathbf{\mu}}^{-1} = \langle V_{3}, V_{4} \rangle,$$ where $V_{3} = U_{\mathbf{\mu}} V_{1}U_{\mathbf{\mu}}^{-1}$ and $V_{4} = U_{\mathbf{\mu}} V_{2}U_{\mathbf{\mu}}^{-1}$. The transformation $U_{\mathbf{\mu}}$ has to satisfy the relationship $U_{\mathbf{\mu}} S_1 U_{\mathbf{\mu}}^{-1} = S_1^{-1}$. As before, this condition fixes $U_{\mathbf{\mu}}$ up to one real parameter $\mathbf{\mu}$ and we can rewrite $U_{\mathbf{\mu}}$ in the form: 
$$U_{\mathbf{\mu}} = \left(
\begin{array}{cc}
\sinh \frac{\mathbf{\mu}}{2} & \cosh\frac{\mathbf{\mu}}{2}\coth \frac{c}{2} \\
-\cosh\frac{\mathbf{\mu}}{2}\tanh \frac{c}{2} & -\sinh\frac{\mathbf{\mu}}{2} \\
\end{array}
\right).$$ Following the notation of the previous section, we have $R_{\mathbf{\mu}} = U_{\mathbf{\mu}} \cdot L$, where $L = \left(
\begin{array}{cc}
0 & -1 \\
1 & 0\\
\end{array}
\right)$. 
Again, the parameter $\mathbf{\mu} \in \R$ has a geometrical interpretation: defining the points $Q_1$, $Q_2$ and $Q_3$ as in the previous section, $\mathbf{\mu}$ is the signed distance of the point $Q_3$ from $Q_1 = i\coth \frac{c}{2}$, where the sign of $\mathbf{\mu}$ is chosen to be positive if moving from $Q_1$ to $Q_3$ takes one in a positive (anti-clockwise) direction around the axis of $S_1$. In fact, since $L(Q_2) = Q_1$, we have that $Q_3 = R_{\mathbf{\mu}}(Q_2) = U_{\mathbf{\mu}}(Q_1)$.

The surface $X$ is obtained by gluing together the two geodesic boundary components $\s_1$ of $X_{1}$ and $\s_2$ of $X_{2}$. This procedure is done by amalgamating the groups $G_{1}$ and $G_{2}$ along the common cyclic subgroup $\langle S_1^{2}\rangle$. We form, in this way, a new Fuchsian group $G$ which is an AFP (short for \textit{amalgamated free product}) of the groups $G_{1}$ and $G_{2}$ along the common cyclic subgroup $\langle S_1^{2}\rangle$:
$$G = G(c, \mathbf{\mu}) = (G_{1}) *_{\langle S_1^{2} \rangle}(G_{2}) = \langle V_{1}, V_{2}, V_{3}, V_{4}| V_{1}V_{2}V_{3}V_{4} = \Id\rangle.$$

  We remark, as before, that the construction also makes sense when one considers complex parameters $\mathbf{\mu} \in \C_{(-\pi, \pi)}$, but the constructed group is not necessarily in $\QF(\Sigma_{0,4})$ for all $\mathbf{\mu} \in \C_{(-\pi, \pi)}$. On the other hand, by Theorem 3.5 of \cite{ser_onk}, there exists $\epsilon > 0$, such that if $|\mathbf{\mu}| < \e$, then the group $G = G(c, \mathbf{\mu})$ is quasi-Fuchsian.

\begin{Remark}\label{four_hyper} 
  If you consider the action of $G$ on $\C \subset \hat \C$ (when $\mathbf{\mu} \in \C_{(-\pi, \pi)}$), you can see that the circle with centre at $-i \coth(\frac{c}{2}) \tan(\frac{\Im \mathbf{\mu}}{2})$ and radius $\coth(\frac{c}{2}) \sec(\frac{\Im \mathbf{\mu}}{2})$ is mapped by $U_\mathbf{\mu}$ to itself. Moreover this circle is mapped to itself under $\langle S_1^2 \rangle$. This circle is a hypercycle of $\Ax(S_1) = \Ax(S_1^2)$. 
\end{Remark}
  
\end{appendix}

\bibliographystyle{amsalpha}
\bibliography{sample1}

\end{document}